\documentclass[review,hidelinks,onefignum,onetabnum]{siamart220329}

\usepackage{lipsum}
\usepackage{amsfonts}
\usepackage{graphicx}
\usepackage{epstopdf}
\usepackage{algorithmic}
\ifpdf
  \DeclareGraphicsExtensions{.eps,.pdf,.png,.jpg}
\else
  \DeclareGraphicsExtensions{.eps}
\fi

\newsiamremark{remark}{Remark}
\newsiamremark{hypothesis}{Hypothesis}
\newsiamremark{example}{Example}
\crefname{hypothesis}{Hypothesis}{Hypotheses}
\newsiamthm{claim}{Claim}

\headers{Least $H^2$ Norm Updating Quadratic Model}{P. Xie and Y. Yuan}

\title{Least $H^2$ Norm Updating Quadratic Interpolation Model Function for Derivative-free Trust-region Algorithms}

\author{Pengcheng Xie\thanks{Institute of Computational Math and Scientific/Engineering Computing, Academy of Mathematics and Systems Science, Chinese Academy of Sciences, and University of Chinese Academy of Sciences, China (\email{xpc@lsec.cc.ac.cn}, \email{yyx@lsec.cc.ac.cn}).}
\and Ya-xiang Yuan\footnotemark[2]
}

\usepackage{amsopn}

\ifpdf
\hypersetup{
  pdftitle={Least \texorpdfstring{$H^2$}{H 2} Norm Updating Quadratic Interpolation Model Function for Derivative-free Trust-region Algorithms},
  pdfauthor={P. Xie and Y. Yuan}
}
\fi

\usepackage{booktabs}
\allowdisplaybreaks[4]
\usepackage{subfigure}

\usepackage{mathptmx}
\DeclareMathAlphabet{\mathcal}{OMS}{cmsy}{m}{n}
\DeclareSymbolFont{largesymbols}{OMX}{cmex}{m}{n}

\newcommand{\x}{\mathbf{x}}
\newcommand{\y}{\mathbf{y}}
\newcommand{\g}{\mathbf{g}}
\newcommand{\dd}{\mathbf{d}}

\newcommand{\e}{\mathbf{e}}
\newcommand{\z}{\mathbf{z}}

\newcommand{\V}{\mathcal{V}}
\newcommand{\C}{\mathcal{C}}
\newcommand{\LL}{\mathcal{L}}
\newcommand{\HH}{\mathcal{H}}
\newcommand{\WW}{\mathcal{W}}

\usepackage{contour} 

\begin{document}

\maketitle

% REQUIRED
\begin{abstract}
Derivative-free optimization methods are numerical methods for optimization problems in which no derivative information is used. Such optimization problems are widely seen in many real applications. One particular class of derivative-free optimization algorithms is trust-region algorithms based on quadratic models given by under-determined interpolation. Different techniques in updating the quadratic model from iteration to iteration will give different interpolation models. In this paper, we propose a new way to update the quadratic model by minimizing the $H^2$ norm of the difference between neighbouring quadratic models. Motivation for applying the $H^2$ norm is given, and theoretical properties of our new updating technique are also presented. Projection in the sense of $H^2$ norm and interpolation error analysis of our model function are proposed. We obtain the  coefficients of the quadratic model function by using the KKT conditions. Numerical results show advantages of our model, and the derivative-free algorithms based on our least $H^2$ norm updating quadratic model functions can solve test problems with fewer function evaluations than algorithms based on least Frobenius norm updating. 
\end{abstract}

% REQUIRED
\begin{keywords}
interpolation, quadratic model, derivative-free, $H^2$ norm, trust-region
\end{keywords}

% REQUIRED
\begin{MSCcodes}
90C56, 90C30, 65K05, 90C90
\end{MSCcodes}

\section{Introduction}
In this paper, we consider the unconstrained derivative-free optimization problem
\begin{equation}\label{dfo-uncon}
\min_{\x \in {\Re}^{n}}\ f(\x),
\end{equation}
where \(f\) is a real-valued function with  first-order or higher-order derivative unavailable. Derivative-free optimization methods are numerical methods for optimization problems in which no derivative information is used. Since only function values are used, derivative-free optimization methods are normally simple, and are easy to be implemented. These methods have broad applications. For example, applications include tuning of algorithmic parameters \cite{Audet:2006:FOA:1237686.1237688}, molecular geometry in biochemistry \cite{doi:10.1002/jcc.540150606},  automatic error analysis \cite{Higham:1993:ODS:153038.153039, Higham2002}, and dynamic pricing \cite{levina2009dynamic}. Optimal design problems in engineering design \cite{booker1998managing, Booker1998-2, Serafini1999}  refer to derivative-free optimization as well, which include wing platform design \cite{Conn016}, aeroacoustic shape design \cite{Conn164, Conn165}, and hydrodynamic design \cite{Conn087}. There are still some disadvantages of derivative-free methods on account of unavailable derivative information. It is  difficult to obtain accurate approximations for the  objective function or its derivative. Besides, the theoretical properties of such methods are usually complicated to analyze. 

We list some types of derivative-free algorithms as examples, and more details are described and discussed in related work \cite{brent2013algorithms, kelley1999iterative, conn2009introduction,zhang2021}. 
One of the types is the line-search type. Some examples are alternating coordinate method, rotating coordinate method  \cite{rosenbrock1960automatic}, conjugate direction method  \cite{powell1964efficient}, finite difference quasi-Newton method \cite{Gill1983}, and general line-search method  \cite{Grippo1988}. There exist methods using approximate Hessian eigenvectors as directions \cite{gratton2016second}, and an indicator for the switch from derivative-free to derivative-based optimization \cite{gratton2017indicator}.

Another type is direct-search method. {\color{black}For example, there is the pattern-search method  \cite{Jeeves, Dennis91,Torczon1991}.} Simplex methods are in this type, which include Nelder-Mead method \cite{Nelder1965}. There are also modified simplex method \cite{Tseng1999},  directional direct-search methods, generating set search method \cite{Kolda2003}, and mesh adaptive direct-search methods with an example NOMAD \cite{nomad}. Direct-search methods with worst-case analysis are considered \cite{dodangeh2016optimal,konevcny2014simple}, and the methods with probabilistic descent are considered by Gratton, Royer, Vicente, and Zhang \cite{gratton-1,gratton2019direct}.

Some derivative-free methods are hybrid. We have implicit filtering method \cite{Kelley2011}, 
adaptive regularized method  \cite{Cartis2012}, Bayesian optimization \cite{Bayes01} and so on. Heuristic algorithms can be derivative-free. For instance, there are simulated annealing method \cite{SimulatedAnnealing01}, algorithms based on artificial neural network \cite{ANN01}, and genetic algorithm \cite{audet2017derivative}.

Closely connected to what we will discuss in this paper, model-based method is another kind of derivative-free methods. Polynomial model methods refer to linear interpolation \cite{Spendley1962}, quadratic interpolation  \cite{Winfield1969,  Winfield1973}, under-determined quadratic interpolation \cite{Powell2003}, and regression \cite{Conn2008b}. Radial basis function interpolation \cite{Mattias2000} is another choice to obtain the model. Probabilistic models can also be used in trust-region methods \cite{bandeira2014convergence,Gratton-2017}. There are also model-based methods for noised problems \cite{troltzsch2012model}.   

Derivative-free algorithms have been designed for problems with special structures as well. For example, the algorithms for solving constrained problems \cite{COBYlA,powell2009bobyqa,gratton2011active}, large-scale nonlinear data assimilation problems \cite{gratton2014derivative}, least-squares minimization \cite{zhang2010derivative, cartis2019derivative} and so on.

The algorithm proposed in this paper is a development of Powell's derivative-free algorithms \cite{Powell2003,NEWUOA,powell2013beyond}. The main idea of Powell's algorithms is to obtain the  quadratic model function by under-determined interpolation at each iteration, for which there is an updating on previous model(s). The unique quadratic model $Q_k$ is obtained at the $k$-th iteration by solving optimization problem
\begin{equation}
\label{objfunction}
  \min_{Q\in\mathcal{Q}}\ \left\|\nabla^{2} Q-\nabla^{2} Q_{k-1}\right\|_{{F}}^{2}, \
  \operatorname{subject\ to}\ Q(\x_i)=f(\x_i),\ \x_i \in \mathcal{X}_{k},
\end{equation}
where \(\mathcal{X}_{k}\) denotes the interpolation set at the $k$-th iteration, and $\mathcal{Q}$ denotes the set of quadratic functions\footnote{Here and below, we use the terminology ``quadratic function'' to refer to the polynomial with integer order that is not larger than 2.}. 
Then the method aims to obtain a new iteration point by minimizing quadratic model $Q_k$ within the trust region. Powell's derivative-free softwares include 
COBYLA \cite{COBYlA}, UOBYQA \cite{UOBYQA}, NEWUOA \cite{NEWUOA}, BOBYQA \cite{powell2009bobyqa}, LINCOA \cite{powell2015fast}.  Besides, Powell's derivative-free optimization solvers with MATLAB and Python interfaces are designed by Zhang\footnote{https://www.pdfo.net}.

There are actually different objective functions  of (\ref{objfunction}) to be chosen, which also generate different models respectively. For
example, \(\left\|\nabla^{2} Q-\nabla^{2} Q_{k-1}\right\|_{{F}}^{2}\)  can be replaced by
\begin{itemize}
\item[-] \( \left\|\nabla^{2} Q\right\|_{{F}}^{2}+\left\|\nabla Q\right\|_{2}^{2} \), proposed by Conn and Toint \cite{conn1996algorithm};
\item[-]  \(\left\|\nabla^{2} Q\right\|_{{F}}^{2}\), proposed by Conn, Scheinberg and Toint \cite{conn1997convergence};
\item[-]  \(\left\|\operatorname{vec}\left(\nabla^{2} Q\right)\right\|_{1}\), proposed by Bandeira, Scheinberg and Vicente \cite{bandeira2012computation}.
\end{itemize}

As shown in the book of Conn, Scheinberg and Vicente \cite{conn2009introduction}, to obtain a fully linear interpolation model (error bound is given by (\ref{fullylin}) in the following), there should be at least $n+1$ interpolation points, and this is expensive in the derivative-free optimization application. Actually, $n+1$ interpolation conditions, referring to $n+1$ equalities constraint conditions, can be relaxed to making the average interpolation error in a region smaller. We apply $H^2$ norm to measure and control the interpolation error.  The least \(H^{2}\) norm updating quadratic model function we proposed can reduce the lower bound of the number of the interpolation points, and can  control the interpolation error locally at the same time.

The rest of this paper is organized as the following. Section \ref{section2} presents some basic results about the \(H^{2}\) norm of a quadratic function. In Section \ref{section3}, motivation of applying the  least \(H^{2}\) norm to update quadratic model function, the projection theory in the sense of \(H^{2}\) norm, and the interpolation error bound are discussed. We propose the least \(H^{2}\) norm updating quadratic model function in Section \ref{section4} with the details of KKT conditions of the  optimization problem to obtain the quadratic model. Furthermore, we  provide more implementation details, including the updating formula of the KKT inverse matrix, and the model improvement step that maximizes the denominator of the updating formula. Numerical results are shown in Section \ref{section7}. At last, a conclusion and some possible future work are given.

\section{\texorpdfstring{$H^2$}{H 2} Norm and Derivative-free Trust-region Algorithms\label{section2}} 
\begin{sloppypar}
We firstly clarify the notation before {\color{black}the} more discussion. A region is defined as 
\(
\Omega=\mathcal{B}_2^{r}(\x_{0}) = \left \{ \x \in {\Re}^{n}: \left\|\x-\x_{0}\right\|_2 \leq r \right\},
\) 
and 
\(
\mathbf{G}=\nabla^2Q \in {\Re}^{n \times n},  \g=\nabla Q(\x_0)  \in {\Re}^{n}, c \in {\Re}
\). 
\(\V_{n}\) denotes the volume of the $n$-dimensional $\ell_2$ unit ball  \(\mathcal{B}_2^{1}(\x_{0})\). 
\(\|\cdot\|_F\) denotes the Frobenius norm, i.e., given the matrix 
\(\mathbf{C}=\left(c_{i j}\right)_{n \times n}\), it holds that \(\left\| \mathbf{C} \right\|_{F} = \bigg( \sum_{i, j}c_{i j}^{2} \bigg)^{\frac{1}{2}}\).
\end{sloppypar}
 
We give the definitions of different kinds of (semi-)norm in the following.

\begin{definition}\label{def42}
Let \(u\) be a function over \(\Omega \subseteq {\Re}^n\), and \(1 \leq p < \infty\). If \(u\) is {2-nd} order differential over \(\Omega\) and \(\frac{\partial^{a} u} {\partial \x^{a}} \in L^2 (\Omega)\) for any natural number \(a\) which satisfies that \(\vert a\vert\leq 2\), then we have
\begin{equation*}
\left\{
\begin{aligned}
    &\| u \|_{H^{0}(\Omega)} = \left(\int_{\Omega} \vert u(\x)\vert^2 d\x  \right)^{\frac{1}{2}},\ \vert u\vert_{H^{1}(\Omega)} = \left(\int_{\Omega} \| \nabla u(\x) \|_2^2 d\x \right)^{ \frac{1}{2}},\\ 
  &\vert u\vert_{H^{2}(\Omega)} = \left(\int_{\Omega} \| \nabla^2 u(\x) \|_F^2 d\x \right)^{\frac{1}{2}}.
  \end{aligned}
  \right.
  \end{equation*}
  Besides, the definition of \(H^2\) norm of the function $u$ is 
  \begin{align*}
      \| u \|_{H^{2}(\Omega)}^2=\| u \|_{H^{0}(\Omega)}^2+\vert u\vert_{H^{1}(\Omega)}^2+\vert u\vert_{H^{2}(\Omega)}^2.
  \end{align*}
  \end{definition}

\(\vert \cdot \vert\) denotes semi-norm, and \(\Vert \cdot \Vert\) denotes norm. Based on Definition \ref{def42}, simple calculations derive the following theorem about $H^2$ norm of a quadratic function. 
\begin{theorem}
\label{theorem1}
Given the quadratic function
\begin{equation}
\label{quadraticQGgc}
Q(\x)=\frac{1}{2}(\x-\x_0)^{\top}\mathbf{G}(\x-\x_0)+(\x-\x_0)^{\top}\g+c,
\end{equation}
it holds that
\begin{equation}\label{calculationh2}
\begin{aligned}
\|Q\|_{H^{2}(\mathcal{B}_2^{r}(\x_0))}^{2}  
=& \V_{n} r^{n} \Bigg[ \left(\frac{r^{4}}{2(n+4)(n+2)}+\frac{r^{2}}{n+2}+1\right)\|\mathbf{G}\|_{F}^{2}+\left(\frac{r^{2}}{n+2}+1\right)\|\g\|_{2}^{2}\\
  & +\frac{{r^{4}}}{4 (n+4)(n+2)} (\operatorname{Tr}(\mathbf{G}))^{2} +\frac{ r^{2}}{n+2} c\operatorname{Tr} (\mathbf{G}) + c^{2}\Bigg],
\end{aligned}
\end{equation}
where $\operatorname{Tr}(\cdot)$ denotes the trace of a matrix. 
\end{theorem}

\begin{proof}
Direct calculations lead to the following facts, and one can obtain more related details in Zhang's thesis \cite{012}.
\begin{align*}
\|Q(\x)\|_{H^{0}\left(\mathcal{B}_{2}^{r}(\x_{0})\right)}^{2}
=& \V_{n} r^{n}\left(\frac{2 r^{4}}{4(n+2)(n+4)}\|\mathbf{G}\|_{F}^{2}+\frac{r^{2}}{n+2}\|\g\|_{2}^{2}\right.\\
&\left.+\frac{r^{4}}{4(n+2) (n+4)} (\operatorname{Tr}(\mathbf{G}))^{2}+\frac{r^{2}}{n+2}c\operatorname{Tr}(\mathbf{G})+c^{2}\right),
\end{align*}
and 
\[
    \vert Q\vert_{H^1 (\mathcal{B}_2^r (\x_0))}^{2} = \V_{n} r^n \left( \frac{r^2}{n+2} \| \mathbf{G} \|_F^2 + \| \g \|_2^2  \right).
\] 
Moreover, it holds that
\(
\vert Q\vert_{H^{2}\left(\mathcal{B}_2^{r}(\x_0)\right)}^{2}
 =\V_{n} r^{n}\|\mathbf{G}\|_{F}^{2}.
\)
Therefore, (\ref{calculationh2}) is proved.
\end{proof}

\begin{remark}
The point $\x_0$ denotes the center point of the region calculated $H^2$ norm on, or we say the base point \cite{NEWUOA}.
\end{remark}
\begin{sloppypar}
Since the interpolation model function in this paper is proposed especially for the derivative-free trust-region algorithm based on interpolation model, we present the framework of model-based derivative-free trust-region algorithms as Algorithm \ref{algo-TR} before giving more details. Notice that, at the criticality step, we accept the model if the interpolation set is well poised in the traditional sense in the derivative-free optimization, but does not request the model to be a fully linear one, when using the least $H^2$ norm updating model in the case that the number of the interpolation points is less than $n+1$. 
\end{sloppypar}

\begin{algorithm}[H]
\caption{Framework of model-based derivative-free trust-region algorithms\label{algo-TR}} 
\begin{algorithmic}[1]
\STATE \textbf{Input}:  the black-box objective function \(f\) and the initial point \(\x_{\text{int}}\).
\STATE \textbf{Output}: the minimum point \({\x}^*\) and the minimum function value  \(f_{\text{opt}}\).
\STATE Initialize and obtain the interpolation set \(\mathcal{X}_0\), the initial quadratic model function \(Q_{0}({\x})\),  the parameter \(\Delta_{0}, \gamma, \varepsilon_c, \mu, \hat{\eta}_1,\hat{\eta}_2\).

\STATE \textbf{Step 1} (\textbf{Criticality step}): Update the model and sample set, and obtain \(Q_{k}\) and \(\mathcal{X}_{k}=\mathcal{X}_{k-1}\cup\{\x_{\text{new}}\}\backslash\{\x_{t}\}\), where $\x_t$ is the farthest point dropped at this step. If $\left\|\g_{k}\right\|>\epsilon_{c}$, then we accept the $\nabla Q_{k}(\x_{\text{opt}})$ and $\Delta_{k}$, where $\x_{\text{opt}}$ is the minimum point in the current iteration. If $\left\|\nabla Q_{k}(\x_{\text{opt}})\right\| \leq \varepsilon_{c}$, then call the model-improvement step to attempt to certify if the model is {\color{black}accepted} on the current trust region, in the case that the model $Q_{k}$ is not certifiably {\color{black}accepted}, or $\Delta_{k}>\mu\left\|\nabla Q_k(\x_{\text{opt}}) \right\|$.

\STATE \textbf{Step 2} (\textbf{Trial step}): 
Solve 
\[
\min_{\dd}  \ Q_k({\x}_{\text{opt}}+ \dd),\  
\text{subject to} \ \| \dd \|_{2} \leq \Delta_k,
\]
and obtain the solution \(\dd_k\).
\STATE \textbf{Step 3} (\textbf{Acceptance of the trial point}): Check if  \(Q_{k}(\x)\) is  a good model. Compute \(f(\x_{\text{opt}}+\dd_{k})\) and define 
\[
\rho_{k}=({f(\x_{\text{opt}})-f(\x_{\text{opt}}+\dd_{k})})/({Q_{k}(\x_{\text{opt}})-Q_{k}(\x_{\text{opt}}+\dd_{k})}).
\]  
If \(\rho_{k} \geq \hat{\eta}_1\), or \(\rho_{k}>0\) and the model is accepted, then \(\x_{k+1}=\x_{\text{opt}}+\dd_{k}\). Otherwise, the model and the iteration remain unchanged, \(\x_{k+1}=\x_{k}\). 
% \State  
\STATE \textbf{Step 4} (\textbf{Model improvement step}): If \(\rho_{k}<\hat{\eta}_1\) and the model is not accepted, then improve the model (details can be seen in Section 10 in the book of Conn, Scheinberg and Vicente \cite{conn2009introduction}). Define \(Q_{k+1}\) and \(\mathcal{X}_{k+1}\) to be the (possibly improved) model and sample set.
\STATE \textbf{Step 5} (\textbf{Trust-region radius updating}): Set \(\Delta_{k+1}\) based on \(\rho_k\) and \(\Delta_k\), e.g., \(\Delta_{k+1}=\frac{1}{\gamma}\Delta_k\), if $\rho_k<\hat{\eta}_1$; \(\Delta_{k+1}=\gamma\Delta_k\), if $\rho_k>\hat{\eta}_2$; \(\Delta_{k+1}=\Delta_k\), in the other cases.

\STATE Increment $k$ by one and go to {\bf Step 1}, until $\Delta_k<\varepsilon_c$ and $\nabla Q_k(\x_{\text{opt}})<\varepsilon_c$.
  \end{algorithmic} 
\end{algorithm} 

More introduction of derivative-free trust-region methods can be seen in the survey of Larson, Menickelly and Wild \cite{larson2019derivative}, the book of Conn, Scheinberg and Vicente \cite{conn2009introduction}, and the book of Audet and Hare \cite{audet2017derivative}. 

\section{Motivation for Applying \texorpdfstring{$H^2$}{H 2} Norm\label{section3}}

We introduce the motivation of applying \(H^2\) norm to obtain the iterative model function of the objective function in this section. The following discovery reminds us to focus on the relationship between using the norm measured at a point and the norm measured by an average value in a region. Details of this discovery is that for arbitrary quadratic function \(Q\), Zhang \cite{Zhang2014} proved that problems
\begin{equation}
\label{P_{1}(sigma)}
\mathop{\min}\limits_{Q\in\mathcal{Q}}\ \Vert\nabla^2 Q\Vert^{2}_{F}+\sigma\Vert \nabla Q(\x_0)\Vert^{2}_{2},\ 
\text{subject to}\ Q(\x)=f(\x),\ \x\in \mathcal{X},
\end{equation}
and 
\begin{equation}
\label{P_{2}(r)}
\mathop{\min}\limits_{Q\in\mathcal{Q}}\ \vert Q\vert_{H^{1}{(\mathcal{B}^{r}_{2}(\x_{0}))}},\ 
\text{subject to}\ Q(\x)=f(\x),\ \x\in \mathcal{X}
\end{equation}
are equivalent to each other, 
where \(\mathcal{X}\) denotes the corresponding interpolation set, and  \(r=\left(\frac{n+2}{\sigma}\right)^{\frac{1}{2}}\) in (\ref{P_{1}(sigma)}) and (\ref{P_{2}(r)}), which shows the connection between (\ref{P_{1}(sigma)}) and (\ref{P_{2}(r)}).

Furthermore, Powell \cite{powell2004least,NEWUOA} proposed obtaining the least Frobenius norm updating quadratic model function by solving 
\begin{equation}\label{Powell-prob}
\min_{Q\in\mathcal{Q}}\ \left\|\nabla^2Q-\nabla^2Q_{k-1}\right\|^2_{F},\ 
\operatorname{subject\ to}\ Q(\x_{i})=f(\x_{i}),\ \x_{i} \in \mathcal{X}_{k},\ i=1,\ldots,m,
\end{equation}
where \(n+2\le m \le \frac{1}{2}(n+1)(n+2)\). Two advantages of choosing the model function above are listed as follows.

One advantage is that the least Frobenius norm updating quadratic model function keeps being fully linear model locally (see details in the book of Conn, Scheinberg and Vicente \cite{conn2009introduction}), i.e., the model function \(Q\) satisfies the error bound
\begin{equation}\label{fullylin}
\|\nabla f(\y)-\nabla Q(\y)\|_2 \leq \kappa_{\text{eg}} \Delta, \text{ and }  
\vert f(\y)-Q(\y)\vert \leq \kappa_{\text{ef}} \Delta^{2},\ \forall\ \y \in \mathcal{B}_2^{\Delta}(\x),
\end{equation}
where $\kappa_{\text{ef}}$ and $\kappa_{\text{eg}}$ are coefficients. 
The error bound above ensures $Q(\x)$ to be a satisfactory approximation of the objective function, and is helpful for the convergence analysis of the trust-region algorithm based on such model.

The other advantage is that we can iteratively make the quadratic model \(Q_k\) closer to a given quadratic objective function \(f\) than \(Q_{k-1}\) by solving Powell's model subproblem (\ref{Powell-prob}), since equality
\begin{equation}
\label{Powell-proj}
\left\|\nabla^2 Q_{k}-\nabla^2 f\right\|^2_F=\left\|\nabla^2 Q_{k-1}-\nabla^2 f\right\|^2_F-\left\|\nabla^2 Q_{k}-\nabla^2 Q_{k-1}\right\|^2_F
\end{equation}
holds if $f$ is a quadratic function. 

However, there is a lower bound of the number of interpolation points or equations, since it only restricts the freedom of the Hessian matrix by solving subproblem (\ref{Powell-prob}). Besides, former models are only obtained by the interpolation on multiple points, but not consider the average value of the objective function or derivatives in a (trust) region. 
Before proposing how to obtain our new choice of model, we present the convexity property of the objective function (\ref{obj}). 

\begin{theorem}\label{strictlyConvex}
The function
\begin{equation}\label{obj}
C_{1}\Vert Q(\x)\Vert^2_{H^{0}(\Omega)}+C_{2}\vert Q(\x)\vert^2_{H^{1}(\Omega)}+C_{3}\vert Q(\x)\vert^2_{H^{2}(\Omega)} 
\end{equation}
is strictly convex as a function of $Q$, where \(C_{1}, C_{2}, C_{3} \geq 0\) and \(C_1+C_2+C_3>0\).
\end{theorem}
\begin{proof}
It holds that 
\[
\begin{aligned}
&C_{1}\Vert Q(\x)\Vert^2_{H^{0}(\Omega)}+C_{2}\vert Q(\x)\vert^2_{H^{1}(\Omega)}+C_{3}\vert Q(\x) \vert^2_{H^{2}(\Omega)}\\
=&C_1\int_{\Omega}\left\vert Q(\x)\right\vert^{2} d \x+C_2\int_{\Omega}\left\Vert \nabla Q(\x)\right\Vert^{2}_{2} d \x+C_3\int_{\Omega}\left\Vert \nabla^2 Q(\x)\right\Vert^{2}_{F} d \x,
\end{aligned}
\]
and we need to prove that the inequality
\begin{equation}\label{inequality-convex}
\begin{aligned}
&\bigg[\mu \left(C_{1}\Vert Q_a(\x)\Vert^2_{H^{0}(\Omega)}+C_{2}\vert Q_a(\x)\vert^2_{H^{1}(\Omega)}+C_{3}\vert Q_a(\x)\vert^2_{H^{2}(\Omega)}\right)\\
&+(1-\mu)\left(C_{1}\Vert Q_b(\x)\Vert^2_{H^{0}(\Omega)}+C_{2}\vert Q_b(\x)\vert^2_{H^{1}(\Omega)}+C_{3}\vert Q_b(\x)\vert^2_{H^{2}(\Omega)}\right)\bigg]\\
&-\bigg[C_1\|\mu Q_{a}(\x)+(1-\mu) Q_{b}(\x)\|^{2}_{H^{0}(\Omega)}+C_2\vert\mu Q_{a}(\x)+(1-\mu) Q_{b}(\x)\vert^{2}_{H^{1}(\Omega)}\\
&+C_3\vert\mu Q_{a}(\x)+(1-\mu) Q_{b}(\x)\vert^{2}_{H^{2}(\Omega)}\bigg]
> 0
\end{aligned}
\end{equation}
holds for any $Q_a, Q_b\in \mathcal{Q}$, and \(0<\mu<1\), \(C_1,C_2,C_3\ge 0\), \(C_1+C_2+C_3>0\).

The left hand side of inequality (\ref{inequality-convex}) can be turned to be
\begin{equation*}
\begin{aligned}
&\left(\mu-\mu^{2}\right) \bigg(C_1\int_{\Omega}\left\vert Q_{a}(\x)-Q_{b}(\x)\right\vert^{2} d \x+C_2\int_{\Omega}\left\Vert \nabla (Q_{a}(\x)- Q_{b}(\x))\right\Vert^{2}_{2} d \x\\
&+C_3\int_{\Omega}\left\Vert \nabla^2 \left(Q_{a}(\x)-Q_{b}(\x)\right)\right\Vert^{2}_{F} d \x\bigg)
> 0.
\end{aligned}
\end{equation*}
Hence,  we proved that \(\|Q(\x)\|_{H^{2}(\Omega)}^{2}\) is strictly convex as a function of $Q$.
% \qed
\end{proof}

Based on Theorem \ref{strictlyConvex}, the quadratic model function \(Q(\x)\), i.e., the solution of the optimization problem 
\begin{equation}\label{H2example}
 % \begin{aligned}
\mathop{\min}\limits_{Q\in\mathcal{Q}}\ \Vert Q-Q_{k-1}\Vert^2_{H^{2}{(\mathcal{B}^{r}_{2}(\x_{0}))}},\ 
\text{subject to}\ Q(\x)=f(\x),\ \x\in \mathcal{X}_k,
% \end{aligned}
\end{equation}
is unique.

\begin{remark}
The solution of optimization problem
\begin{equation}\label{weightproblem}
 \begin{aligned}
&\ \mathop{\min}\limits_{Q\in\mathcal{Q}}\ C_1\Vert Q-Q_{k-1}\Vert^2_{H^{0}{(\mathcal{B}^{r}_{2}(\x_{0}))}}+C_2\vert Q-Q_{k-1}\vert^2_{H^{1}{(\mathcal{B}^{r}_{2}(\x_{0}))}} +C_3\vert Q-Q_{k-1}\vert^2_{H^{2}{(\mathcal{B}^{r}_{2}(\x_{0}))}}\\
&\text{subject to}\ Q(\x)=f(\x),\ \x\in \mathcal{X}_k,
\end{aligned}
\end{equation}
is unique, where $C_1,C_2,C_3\ge 0$ and $C_1+C_2+C_3>0$. 
\end{remark}

It is obviously to see that the \(H^{1}\) semi-norm of the quadratic model function on \(\mathcal{B}_{2}^{r}(\x_{0})\) refers to the average value of the first-order derivative of the quadratic model function on the region \(\mathcal{B}_{2}^{r}(\x_{0})\), and the \(H^{2}\) semi-norm on \(\mathcal{B}_{2}^{r}(\x_{0})\) refers to the average value of the second-order derivative of the quadratic model function on the region \(\mathcal{B}_{2}^{r}(\x_{0})\). In fact, the aim to minimize the interpolation error of the quadratic model function value is also important. Therefore, according to the projection theory in the following section, it is reasonable for us to include \(H^{0}\) norm, or we say \(L^{2}\) norm, in the objective function of (\ref{weightproblem}). {\color{black} Minimizing \(L^{2}\) norm can reduce the interpolation error of the function value instead of using the interpolation condition $f(\x_i)=Q(\x_i)$, and then relax the lower bound of the number of the interpolation points.} A (weighted) summation of \(H^{0}\) norm, \(H^{1}\) semi-norm and \(H^{2}\) semi-norm of \(Q(\x)\) on the region \(\mathcal{B}_{2}^{r}(\x_{0})\) indicates our aim to minimize the average function error, the average error of the first-order derivative  and the average error of the second-order derivative of the model function at the same time. We believe that in this way $Q(\x)$ will be a better model. It can reduce the lower bound of the number of the interpolation points, i.e., $m$ can be less than $n+1$. Numerical results also support our choice of $H^2$ norm.

The following two subsections will show that least $H^2$ norm updating quadratic model holds similar projection property as (\ref{Powell-proj}), and it has interpolation error upper bound locally. 

\subsection{Projection in the sense of \texorpdfstring{$H^2$}{H 2} norm}
We now introduce the projection in the sense of \(H^2\) norm, and we can obtain the similar projection  result as (\ref{Powell-proj}). We present details in the following theorem.

\begin{theorem}\label{projection}
Let \(Q_{k}\) be the solution of problem (\ref{H2example}). If $f$ is quadratic, then
\begin{equation}\label{projH}
\left\|Q_{k}-f\right\|_{H^{2}(\mathcal{B}_2^r(\x_0))}^{2}=\left\|Q_{k-1}-f\right\|_{H^{2}(\mathcal{B}_{2}^{r}(\x_0))}^{2}-\left\|Q_{k}-Q_{k-1}\right\|_{H^{2} (\mathcal{B}_2^r(\x_0))}^{2}.
\end{equation}
\end{theorem}

\begin{proof}
At the \(k\)-th iteration, we denote \(Q_{k}+\xi \left(Q_{k}-f\right)\) as \(Q_{\xi}\),  for arbitrary \(\xi \in {\Re}\). Then \(Q_{\xi}\) is an interpolation function, and it satisfies the interpolation conditions in (\ref{H2example}). Therefore,  it holds that 
\(
\varphi(\xi)=\left\|Q_{\xi}-Q_{k-1}\right\|_{H^{2}(\mathcal{B}_2^r(\x_0))}^{2}
\) 
has the minimum value when \(\xi=0\), according to the optimality of \(Q_{k}\). In addition, we have
\begin{equation}
\label{above}
\begin{aligned}
\varphi(\xi)
=&\left\|Q_{k}+\xi \left(Q_{k}-f\right)-Q_{k-1}\right\|_{H^{2}(\mathcal{B}_2^r(\x_0))}^{2} \\
=&\xi^{2}\left\|Q_{k}-f\right\|_{H^{2}(\mathcal{B}_2^r(\x_0))}^{2} +\left\|Q_{k}-Q_{k-1}\right\|_{H^{2}(\mathcal{B}_2^r(\x_0))}^{2} \\
&+2 \xi \bigg\{\int_{\mathcal{B}_2^r(\x_0)}\left[\left(Q_{k}-Q_{k-1}\right)(\x)\right] \cdot\left[\left(Q_{k}-f\right)(\x)\right] d \x\\
&+\int_{\mathcal{B}_2^r(\x_0)}\left[\nabla(Q_{k}-Q_{k-1})(\x)\right]^{\top}\left[\nabla(Q_{k}-f)(\x)\right] d \x \\
&+\int_{\mathcal{B}_2^r(\x_0)} (1,\ldots,1)
\left[\nabla^{2}(Q_{k}-Q_{k-1})(\x)\right] \circ \left[\nabla^{2}(Q_{k}-f)(\x)\right](1,\ldots,1)^{\top} d \x\bigg\},
\end{aligned}
\end{equation}
where the symbol \(\circ\) denotes Hadamard product. 
Therefore the term in the last brace of (\ref{above}) is \(0\). Considering \(\varphi(-1)\), we can obtain the conclusion of this theorem.
\end{proof}

According to (\ref{projH}), we can obtain the projection in the sense of \(H^2\) norm
\begin{equation*}
\left\|Q_{k}(\x)-f(\x)\right\|_{H^{2}(\mathcal{B}_2^r(\x_0))} \leq\left\|Q_{k-1}(\x)-f(\x)\right\|_{H^{2}(\mathcal{B}_2^r(\x_0))}.
\end{equation*}
The model function \(Q_{k}\) has more accurate function value and derivatives than \(Q_{k-1}\), unless \(Q_{k}=Q_{k-1}\). We can directly obtain the following corollary from Theorem \ref{projection}.

\begin{corollary}
Let \(Q_{k}\) be the solution of problem (\ref{weightproblem}). If $f$ is quadratic, then 
\begin{equation*}
\begin{aligned}
&C_1\left\|Q_{k}-f\right\|_{H^{0}(\mathcal{B}_2^r(\x_0))}^{2}
+
C_2\left\vert Q_{k}-f\right\vert_{H^{1}(\mathcal{B}_2^r(\x_0))}^{2}
+
C_3\left\vert Q_{k}-f\right\vert_{H^{2}(\mathcal{B}_2^r(\x_0))}^{2}\\
=&C_1\left\|Q_{k-1}-f\right\|_{H^{0}(\mathcal{B}_2^r(\x_0))}^{2}
+
C_2\left\vert Q_{k-1}-f\right\vert_{H^{1}(\mathcal{B}_2^r(\x_0))}^{2}
+
C_3\left\vert Q_{k-1}-f\right\vert_{H^{2}(\mathcal{B}_2^r(\x_0))}^{2}\\
-&C_1\left\|Q_{k}-Q_{k-1}\right\|_{H^{0}(\mathcal{B}_2^r(\x_0))}^{2}
-
C_2\left\vert Q_{k}-Q_{k-1}\right\vert_{H^{1}(\mathcal{B}_2^r(\x_0))}^{2}
-
C_3\left\vert Q_{k}-Q_{k-1}\right\vert_{H^{2}(\mathcal{B}_2^r(\x_0))}^{2},
\end{aligned}
\end{equation*}
where $C_1,C_2,C_3\ge 0$. 
\end{corollary}

\begin{proof}
The proof is similar to that of Theorem \ref{projection}.
\end{proof}

\subsection{Error analysis of the model}
The model function approximating the objective function is significant to our model-based optimization algorithm. 
We present the error analysis of the interpolation model. We give two lemmas firstly in the following. 

\begin{lemma} [Interpolation inequality]
\label{interpolationinequ}
 Assume that \(1 \leq c_1 \leq c_2 \leq c_3 \leq \infty\), and 
\(
\frac{1}{c_2}=\frac{\theta}{c_1}+\frac{(1-\theta)}{c_3}
\). Suppose also that \(u \in L^{c_1}(\Omega) \cap L^{c_3}(\Omega)\). Then \(u \in L^{c_2}(\Omega)\), and
\(
\|u\|_{L^{c_2}(\Omega)} \leq\|u\|_{L^{c_1}(\Omega)}^{\theta}\|u\|_{L^{c_3}(\Omega)}^{1-\theta}.
\)
\end{lemma}

\begin{proof}
Proof details are in the book of Evans \cite{evans10}.
\end{proof}

\begin{definition}[The Sobolev space]
The Sobolev space $\WW^{k, p}(\Omega)$ 
consists of all locally summable functions $u: \Omega \rightarrow \Re$ such that for each multiindex $\alpha$ with $\vert\alpha\vert \leq k, D^\alpha u$ exists in the weak sense and belongs to $L^p(\Omega)$\footnote{More details of the locally summable functions and the weak derivatives can be seen in the book of Evans \cite{evans10}}. If $p=2$, we usually write $
\HH^k(\Omega)=\WW^{k, 2}(\Omega). 
$ 
\end{definition}

\begin{lemma}\label{normlemma}
Suppose that $u\in \HH^1(\Omega)$, and 
\(
\left\vert \partial_{i} u\right\vert  \leq M_{1}, \  i=1,\ldots, n
\). 
For \(\forall\ \y\in \Omega:=\mathcal{B}_{2}^{r}(\x_0)\), it holds that
\begin{equation}\label{1lemma}
\begin{aligned}
\vert u(\y)\vert \le &\V_{n}^{-\frac{1}{2}}  r^{-\frac{n}{2}}  \|u\|_{L^{2}(\Omega)}+\V_{n}^{-\frac{1}{p}}r^{\frac{n}{q}+1-n} n^{\frac{1}{q}-\frac{1+\theta}{2}}\left(n+q-nq\right)^{-\frac{1}{q}} M_{1}^{1-\theta}\vert u\vert_{H^{1}(\Omega)}^{\theta},
\end{aligned}
\end{equation}
where \(\frac{1}{p}+\frac{1}{q}=1\), \(\theta=\frac{2}{p} \leq 1\), \(p>n\), {\color{black}
\(
\mathcal{B}_2^{r}(\x_{0}) = \left \{ \x \in {\Re}^{n}: \left\|\x-\x_{0}\right\|_2 \leq r \right\},
\) 
and 
\(\V_{n}\) denotes the volume of the $n$-dimensional $\ell_2$ unit ball  \(\mathcal{B}_2^{1}(\x_{0})\)}.
\end{lemma}

\begin{proof}
We extend the function $u$ outside the region $\Omega$, such that $u(\y)=0$ for any \(\y \notin \mathcal{B}_{2}^{r}\left(\x_0\right)\). For any \(\y \in \mathcal{B}_{2}^{r}\left(\x_0\right)\), the following conclusion holds.
\begin{equation}
\label{upperbound}
\begin{aligned}
\left\vert u(\y) \right\vert & \leq \left\vert u(\y)-\frac{1}{\left\vert \mathcal{B}_{2}^{r}(\y)\right\vert} \int_{\mathcal{B}_{2}^{r}(\y)} u(\x) d \x\right\vert+\frac{1}{\left\vert \mathcal{B}_{2}^{r}(\y)\right\vert} \left\vert\int_{\mathcal{B}_{2}^{r}(\y)} u(\x) d \x \right\vert \\
& \leq \frac{1}{\left\vert\mathcal{B}_{2}^{r}(\y)\right\vert}
\left(\int_{\mathcal{B}_{2}^{r}(\y)}\left\vert u(\y)-u(\x)\right\vert d \x+\left\vert\int_{\mathcal{B}_{2}^r(\y)} u(\x) d \x \right\vert\right),
\end{aligned}
\end{equation}
where \(\left\vert \mathcal{B}_{2}^{r}(\y)\right\vert\) denotes the volume of the ball \( \mathcal{B}_{2}^{r}(\y)\). Besides, there are upper bounds for the two terms appearing in the right hand side of (\ref{upperbound}). Based on H\"{o}lder's inequality, we can obtiain
\begin{equation*}
\begin{aligned}
\left\vert\int_{\mathcal{B}_{2} ^{r}(\y)} u(\x) d \x \right\vert & \leq \left(\int_{\mathcal{B}_{2} ^{r}(\y)}(u(\x))^{2} d \x\right)^{\frac{1}{2}} \left(\int_{\mathcal{B}_{2}^{r}(\y)} 1^{2} d \x\right)^{\frac{1}{2}}\\
&\leq \left\vert\mathcal{B}_{2}^{r}(\y)\right\vert^{\frac{1}{2}}  \left\Vert u\right\Vert_{L^2(\Omega)}\\
&=\V_{n}^{\frac{1}{2}}r^{\frac{n}{2}}  \left\Vert u\right\Vert_{L^2(\Omega)}.
\end{aligned}
\end{equation*}
According to the proof of Morrey's inequality in the book of Evans \cite{evans10}, it holds that
\begin{equation*}
\begin{aligned}
\int_{\mathcal{B}_{2}^{r}(\y)}\left\vert u(\y)-u(\x)\right\vert d \x 
\leq & \frac{r^{n}}{n} \int_{\mathcal{B}_{2}^{r}(\y)} \frac{\left\Vert\nabla u(\x)\right\Vert_{2}}{\left\Vert \y-\x\right\Vert_{2}^{n-1}} d \x \\
\leq & \frac{r^{n}}{n}\left(\int_{\mathcal{B}_{2}^{r}(\y)}\left\Vert\nabla u(\x)\right\Vert_{2}^{p} d \x\right)^{\frac{1}{p}}
\left(\int_{\mathcal{B}_{2}^{r}(\y)} \frac{1}{\left\Vert \y-\x\right\Vert_{2}^{(n-1)q}} d \x\right)^{\frac{1}{q}},
\end{aligned}
\end{equation*}
where \(\frac{1}{p}+\frac{1}{q}=1\), and \((n-1)(q-1)\in(0,1)\), which always holds if $p>n$. 
\begin{equation*}
 \begin{aligned}
\left(\int_{\mathcal{B}_{2}^{r}(\y)} \frac{1}{\left\Vert \y-\x\right\Vert_{2}^{(n-1)q}} d \x\right)^{\frac{1}{q}}
=&\left(\int_{\mathcal{B}_{2}^{r}(\mathbf{0})} \frac{1}{\z^{(n-1)q}} d \z\right)^{\frac{1}{q}} \\
=&\V_{n}^{\frac{1}{q}}n^{\frac{1}{q}}\left(n+q-nq\right)^{-\frac{1}{q}}r^{\frac{n}{q}+1-n}.
 \end{aligned}
\end{equation*}
Besides, Lemma \ref{interpolationinequ} helps us to obtain the result that
\begin{align*}
\left(\int_{\mathcal{B}_{2}^{r}(\y)} \Vert \nabla u(\x)\Vert_{2}^{p} d \x\right)^{\frac{1}{p}} 
=&\left\Vert \left(\sum_{i=1}^n\vert \partial_i u(\x)\vert^2\right)^{\frac{1}{2}}\right\Vert_{L^p(\mathcal{B}_{2}^{r}(\y))}\\
\leq &\left\Vert \left(\sum_{i=1}^n\vert \partial_i u(\x)\vert^2\right)^{\frac{1}{2}}\right\Vert_{L^p(\Omega)}\\
\leq& \left\Vert \left(\sum_{i=1}^n\vert \partial_i u(\x)\vert^2\right)^{\frac{1}{2}}\right\Vert_{L^2(\Omega)}^{\theta}\left\Vert \left(\sum_{i=1}^n\vert \partial_i u(\x)\vert^2\right)^{\frac{1}{2}}\right\Vert_{L^\infty(\Omega)}^{1-\theta} \\
\leq& n^{\frac{1-\theta}{2}}\left\vert u\right\vert_{H^1(\Omega)}^{\theta}M_1^{1-\theta}, 
\end{align*}
where \(\theta=\frac{2}{p} \leq 1\), and \(p>n\).
Then it holds that
\begin{equation*}
\begin{aligned}
\int_{\mathcal{B}_{2}^{r}(\y)}\left\vert u(\y)-u(\x)\right\vert d \x
\leq &\frac{r^{n}}{n}n^{\frac{1-\theta}{2}}\left\vert u\right\vert_{H^1(\Omega)}^{\theta}M_1^{1-\theta}\V_{n}^{\frac{1}{q}}n^{\frac{1}{q}}\left(n+q-nq\right)^{-\frac{1}{q}}r^{\frac{n}{q}+1-n}\\
=&\V_{n}^{\frac{1}{q}}{r^{\frac{n}{q}+1}}n^{\frac{1}{q}-\frac{1+\theta}{2}}\left(n+q-nq\right)^{-\frac{1}{q}}M_1^{1-\theta}\left\vert u\right\vert_{H^1(\Omega)}^{\theta}.
\end{aligned}
\end{equation*}
Hence we have
\begin{equation*}
\begin{aligned}
\vert u(\y)\vert 
\le &\frac{1}{\V_{n} r^{n}}\left[{r^{\frac{n}{q}+1}}n^{\frac{1}{q}-\frac{1+\theta}{2}}\left(n+q-nq\right)^{-\frac{1}{q}}\vert u\vert_{H^{1}(\Omega)}^{\theta}M_1^{1-\theta}\V_{n}^{\frac{1}{q}}+\V_{n}^{\frac{1}{2}}  r^{\frac{n}{2}}  \|u\|_{L^{2}(\Omega)}\right]\\
=&\V_{n}^{-\frac{1}{2}}  r^{-\frac{n}{2}}  \|u\|_{L^{2}(\Omega)}+\V_{n}^{\frac{1}{q}-1}r^{\frac{n}{q}+1-n} n^{\frac{1}{q}-\frac{1+\theta}{2}}\left(n+q-nq\right)^{-\frac{1}{q}} M_{1}^{1-\theta}\vert u\vert_{H^{1}(\Omega)}^{\theta}.
\end{aligned}
\end{equation*}

Therefore, (\ref{1lemma}) holds.
\end{proof}

The following theorem illustrates the relationship between the $H^2$ norm of the function in a given region and the absolute value of the function value and gradient norm at a point.

\begin{theorem}\label{THM-error}
Let {\color{black}\(u \in \HH^{2}(\Omega)\)}, where \(\Omega=\mathcal{B}_2^{r}\left(\x_{0}\right)\). Suppose there exist \(M_1, M_2\) such that
\(
\left\vert \partial_{i} u\right\vert  \leq M_{1}, \ 
\left\vert \partial_{i j}^{2} u\right\vert  \leq M_{2}, \  i, j=1,\ldots, n,
\) 
then, for $\forall\ x \in \Omega$, we have
\begin{align}
&\vert u(\x)\vert\le \V_{n}^{-\frac{1}{2}}  r^{-\frac{n}{2}}  \|u\|_{L^{2}(\Omega)}+\V_{n}^{-\frac{1}{p}}r^{\frac{n}{q}+1-n} n^{\frac{1}{q}-\frac{1+\theta}{2}}\left(n+q-nq\right)^{-\frac{1}{q}} M_{1}^{1-\theta}\vert u\vert_{H^{1}(\Omega)}^{\theta},\label{conclusion-1}\\
&\|\nabla u(\x)\|_{2}\le n^{\frac{1}{2}}\V_{n}^{-\frac{1}{2}}  r^{-\frac{n}{2}}  \vert u\vert_{H^{1}(\Omega)}+\V_{n}^{-\frac{1}{p}}r^{\frac{n}{q}+1-n} n^{\frac{1}{q}-\frac{\theta}{2}}\left(n+q-nq\right)^{-\frac{1}{q}} M_{2}^{1-\theta}\vert u\vert_{H^{2}(\Omega)}^{\theta},\label{conclusion-2}
\end{align}
where \(\frac{1}{p}+\frac{1}{q}=1\), \(\theta=\frac{2}{p} \leq 1\), \(p>n\), {\color{black}
\(
\mathcal{B}_2^{r}(\x_{0}) = \left \{ \x \in {\Re}^{n}: \left\|\x-\x_{0}\right\|_2 \leq r \right\},
\) 
and 
\(\V_{n}\) denotes the volume of the $n$-dimensional $\ell_2$ unit ball  \(\mathcal{B}_2^{1}(\x_{0})\)}.
\end{theorem}

\begin{proof}
 Lemma \ref{normlemma} directly proves (\ref{conclusion-1}). For \(\forall\ \x\in \Omega:=\mathcal{B}_{2}^{r}(\x_0)\), it holds that
\begin{equation*}
\vert \partial_i u(\x)\vert \le \V_{n}^{-\frac{1}{2}}  r^{-\frac{n}{2}}  \vert u\vert_{H^{1}(\Omega)}+\V_{n}^{-\frac{1}{p}}r^{\frac{n}{q}+1-n} n^{\frac{1}{q}-\frac{1+\theta}{2}}\left(n+q-nq\right)^{-\frac{1}{q}} M_{2}^{1-\theta}\vert u\vert_{H^{2}(\Omega)}^{\theta},
\end{equation*}
where \(\frac{1}{p}+\frac{1}{q}=1\), \(\theta=\frac{2}{p} \leq 1\), and \(p>n\).
Therefore, it follows that
\[
\begin{aligned}
\|\nabla u\|_{2}
&\le {n}^{\frac{1}{2}}\max_{i=1,\ldots,n}\left\|\partial_{i} u\right\|_{L^{\infty}(\Omega)}\\
\le & n^{\frac{1}{2}}\left[\V_{n}^{-\frac{1}{2}}  r^{-\frac{n}{2}}  \vert u\vert_{H^{1}(\Omega)}+\V_{n}^{-\frac{1}{p}}r^{\frac{n}{q}+1-n} n^{\frac{1}{q}-\frac{1+\theta}{2}}\left(n+q-nq\right)^{-\frac{1}{q}} M_{2}^{1-\theta}\vert u\vert_{H^{2}(\Omega)}^{\theta}\right],
\end{aligned}
\]
which gives (\ref{conclusion-2}). 
\end{proof}

According to Theorem \ref{THM-error}, the following corollary holds naturally. It can be observed that reducing the $H^2$ norm of the function $u$ will also reduce the Euclidean norms of the function value and  gradient vector of $u$, which gives the relationship of the norm based on integral in some region and the norm at given point.

\begin{corollary}\label{Q-ferror}
Given the objective function \(f \in \C^{3}(\mathcal{B}_2^r(\x_0))\) and its quadratic model function \(Q\). Suppose that \(
\left\vert \partial_{i} (Q-f)\right\vert  \leq M_{1},\ 
\left\vert \partial_{i j}^{2} (Q-f)\right\vert  \leq M_{2}  
\),  
for \(i,j=1, \ldots, n\). Then, for $\forall\ \x \in \mathcal{B}_2^r(\x_0)$, it holds that
\begin{align*}
\vert Q(\x)-f(\x)\vert
\le& \V_{n}^{-\frac{1}{2}}  r^{-\frac{n}{2}}  \|Q-f\|_{L^{2}(\Omega)}\\
& +\V_{n}^{-\frac{1}{p}}r^{\frac{n}{q}+1-n} n^{\frac{1}{q}-\frac{1+\theta}{2}}\left(n+q-nq\right)^{-\frac{1}{q}} M_{1}^{1-\theta}\vert Q-f\vert_{H^{1}(\Omega)}^{\theta},\\
\|\nabla Q(\x)-\nabla f(\x)\|_{2}
\le& 
n^{\frac{1}{2}}\V_{n}^{-\frac{1}{2}}  r^{-\frac{n}{2}}  \vert Q-f\vert_{H^{1}(\Omega)}\\
&+\V_{n}^{-\frac{1}{p}}r^{\frac{n}{q}+1-n} n^{\frac{1}{q}-\frac{\theta}{2}}\left(n+q-nq\right)^{-\frac{1}{q}} M_{2}^{1-\theta}\vert Q-f\vert_{H^{2}(\Omega)}^{\theta},
\end{align*}
where \(\frac{1}{p}+\frac{1}{q}=1\), \(\theta=\frac{2}{p} \leq 1\), \(p>n\), {\color{black}
\(
\mathcal{B}_2^{r}(\x_{0}) = \left \{ \x \in {\Re}^{n}: \left\|\x-\x_{0}\right\|_2 \leq r \right\},
\) 
and 
\(\V_{n}\) denotes the volume of the $n$-dimensional $\ell_2$ unit ball  \(\mathcal{B}_2^{1}(\x_{0})\)}.
\end{corollary}

\begin{proof}
It is easy to see that the corollary is true by substituting $u$ by $Q-f$ in Theorem \ref{THM-error}. 
\end{proof}

The error analysis in Corollary \ref{Q-ferror} theoretically supports the approximation property of least $H^2$ norm updating quadratic model. Therefore, obtaining the model function by solving subproblem (\ref{H2example}) or subproblem (\ref{weightproblem}) can relax the minimum requirement of the number of interpolation conditions, and it allows us to use fewer interpolation points. At the same time, the relaxed constraint, i.e., making the $H^2$ norm (with weight coefficients) of \(Q-f\) smaller, can still keep good approximation property and error bound.

\section{Least \texorpdfstring{$H^2$}{H 2} Norm Updating Quadratic Model\label{section4}}

In this section, we propose the method obtaining the least $H^2$ norm updating quadratic model according to KKT conditions. Theorem \ref{theorem1} and its proof inspire us to obtain parameters of the quadratic model function at the $k$-th step by solving problem
\begin{equation}
\label{lagobj}
\begin{aligned}
\min_{c,g,\mathbf{G}}\ &\eta_{1}\|\mathbf{G}\|_{F}^{2}+\eta_{2} \| \g\|_{2}^{2}+\eta_{3}(\operatorname{Tr} (\mathbf{G}))^{2}+\eta_{4} \operatorname{Tr} (\mathbf{G})c+\eta_5c^{2}\\
\text{subject to}\ &c+\g^{\top}\left(\x_{i}-\x_{0}\right)+\frac{1}{2}\left(\x_{i}-\x_{0}\right)^{\top} \mathbf{G} \left(\x_{i}-\x_{0}\right)=f(\x_{i})-Q_{k-1}(\x_i), i=1, \ldots, m,
\end{aligned}
\end{equation}
where \(\mathbf{G}^{\top}=\mathbf{G}\), and the solution of (\ref{lagobj}) are the coefficients of the difference function of the model functions, i.e., $Q_k-Q_{k-1}$. The radius $r$ for calculating $H^2$ norm will be given in Section \ref{section7} for implementation.  Points $\x_1,\ldots,\x_m$ denote the interpolation points at the $k$-th iteration for simplicity. 
We directly consider the generated weighted objective function with weight coefficients $C_1$, $C_2$ and $C_3$. 
Parameters \(\eta_1,\eta_2,\eta_3,\eta_4,\eta_5\) satisfy that
\begin{equation}\label{eta}
\left\{
\begin{aligned}
\eta_1&=C_1\frac{r^{4}}{2(n+4)(n+2)}+C_2\frac{r^{2}}{n+2}+C_3,\
\eta_2=C_1\frac{r^{2}}{n+2}+C_2,\\ 
\eta_3&=C_1\frac{r^{4}}{4(n+4)(n+2)},\ 
\eta_4=C_1\frac{r^{2}}{n+2},\ 
\eta_5=C_1.
\end{aligned}
\right.
\end{equation}

\begin{sloppypar}
Then the Lagrange function is 
\begin{equation}
\begin{aligned}\label{lagfunction}
&\mathcal{L}(c, \g, \mathbf{G})  
= \eta_{1}\|\mathbf{G}\|_{F}^{2}+\eta_{2}\|\g\|_{2}^{2}+\eta_{3}(\operatorname{Tr} (\mathbf{G}))^{2} +\eta_{4}\operatorname{Tr} (\mathbf{G})c+\eta_5c^{2} \\
 &\ \ -\sum_{l=1}^{m} \lambda_{l}\left[c+\g^{\top}\left(\x_{l}-\x_{0}\right)+\frac{1}{2}\left(\x_{l}-\x_{0}\right)^{\top} \mathbf{G} \left(\x_{l}-\x_{0}\right) {-f(\x_l)}+Q_{k-1}(\x_l) \right].
\end{aligned}
\end{equation}
Let \(T\) denotes \(\operatorname{Tr} (\mathbf{G})\). KKT conditions of problem (\ref{lagobj}) include that
\begin{align}
0=\frac{\partial \LL(c,\g,\mathbf{G})}{\partial c}&=2 \eta_5c+\eta_4T-\sum_{l=1}^{m} \lambda_{l}, \label{KKT-1}\\
\mathbf{0}_{n}=\frac{\partial \LL(c,\g,\mathbf{G})}{\partial \g}&=2 \eta_{2} \g-\sum_{l=1}^{m} \lambda_l \left(\x_{l}-\x_{0}\right),\label{KKT-2}
\end{align}
where $\mathbf{0}_n=(0,\ldots,0)^{\top}\in{\Re}^{n}$, and the other equations of KKT conditions come from the following. Differentiating \(\LL(c, \g, \mathbf{G})\)  with respect to the elements of \(\mathbf{G}\), we can obtain that
\begin{equation*}
\begin{aligned}
2 \eta_{1} \mathbf{G}-\frac{1}{2}\sum_{l=1}^{m} \lambda_{l}\left(\x_{l}-\x_{0}\right)\left(\x_{l}-\x_{0}\right)^{\top} +2 \eta_{3} \operatorname{Diag}\left\{T, \ldots, T\right\}+\eta_{4}c \mathbf{I}=\mathbf{0}_{nn},
\end{aligned}
\end{equation*}
where $\mathbf{I}\in{\Re}^{n\times n}$ is the identity matrix, and $\mathbf{0}_{nn}\in{\Re}^{n\times n}$ is the zero matrix. Thus 
\begin{equation}
\label{Q(x)G}
\begin{aligned}
2 \eta_{1} \mathbf{G} & =\frac{1}{2}\sum_{l=1}^{m} \lambda_{l}\left(\x_{l}-\x_{0}\right)\left(\x_{l}-\x_{0}\right)^{\top} -\left(2\eta_{3} T+\eta_{4}c\right) \mathbf{I}.
\end{aligned}
\end{equation} 
\end{sloppypar}

We let \(\mathbf{G}\) multiply  \(\left(\x_{i}-\x_{0}\right)^{\top}\) and \(\left(\x_{i}-\x_{0}\right)\), and then we can obtain 
\begin{equation}
\label{also-KKT}
\begin{aligned}
 &2\eta_{1}\left(\x_{i}-\x_{0}\right)^{\top} \mathbf{G}\left(\x_{i}-\x_{0}\right)\\
 =& \frac{1}{2}\sum_{l=1}^{m} \lambda_{l}\left[\left(\x_{i}-\x_{0}\right)^{\top}\left(\x_{l}-\x_{0}\right)\right]^{2} -\left(2\eta_{3}T+\eta_{4}c\right)\left\|\x_{i}-\x_{0}\right\|_2^2,\ 1\le i\le m.
\end{aligned}
\end{equation}
Besides, we let \(\mathbf{G}\) multiply \(\e_{j}^{\top}\) and \(\e_{j}\). The vector {\color{black}\(\e_{j}\in\Re^n\)} is \((0,\ldots,0,1,0,\ldots,0)^{\top}\), of which only the $j$-th element is 1. Then we can obtain 
\begin{equation}
\label{ejGej}
\begin{aligned}
 2\eta_{1}\e_j^{\top} \mathbf{G}\e_j = &\frac{1}{2}\sum_{l=1}^{m} \lambda_{l}\left[\e_j^{\top}\left(\x_{l}-\x_{0}\right)\right]^{2}-\left(2\eta_{3}T+\eta_{4}c\right)\e_j^{\top}\e_j,\ 1\le j\le n.
\end{aligned}
\end{equation}
If we calculate the summation of (\ref{ejGej}) by \(j\) from \(1\) to \(n\), we can obtain 
\begin{equation*}
\begin{aligned}
2\eta_{1}T = &\frac{1}{2}\sum_{l=1}^{m} \sum_{j=1}^{n}\lambda_{l}\left[\e_j^{\top}\left(\x_{l}-\x_{0}\right)\right]^{2}-
n{\left(2\eta_{3}T+\eta_{4}c\right)},
\end{aligned}
\end{equation*}
and then 
\begin{equation}\label{othern}
0=\frac{1}{2}\sum_{l=1}^{m} \lambda_{l}\left\| \x_{l}-\x_{0}\right\|_2^{2}-(2n\eta_{3}+2\eta_1)T-n\eta_{4}c. 
\end{equation}
Hence we obtain the expression of $T$, i.e.,
\begin{equation}
\label{Tiswhat}
T=\frac{1}{2\left(2 n \eta_{3}+2 \eta_{1}\right)} \sum_{l=1}^{m} \lambda_{l}\left\| \x_{l}-\x_{0}\right\|_2^2-\frac{n \eta_{4}}{2 n \eta_{3}+2 \eta_{1}} c.
\end{equation}
\begin{sloppypar}
Combining with the constraints in (\ref{lagobj}), we obtain the following constraints of {\color{black}$\contour[2]{black}{$\lambda$}=(\lambda_1,\ldots,\lambda_m)^{\top},c,\g$,} 
\end{sloppypar}
\begin{equation}
\begin{aligned}
0=&2 \eta_{5} c+\frac{\eta_{4}}{4 n \eta_{3}+4 \eta_{1}} \sum_{l=1}^{m} \lambda_{l}\left\|\x_{l}-\x_{0}\right\|_2^2-\frac{n \eta_{4}^{2}}{2 n \eta_{3}+2 \eta_{1}} c-\sum_{l=1}^{m} \lambda_{l},\\
\mathbf{0}_{n}=&2 \eta_{2} \g-\sum_{l=1}^{m} \lambda_{l}\left(\x_{l}-\x_{0}\right),\\
f(\x_i)=&\frac{1}{8 \eta_{1}} \sum_{l=1}^{m} \lambda_{l} \left[\left(\x_{i}-\x_{0}\right)^{\top}\left(\x_{l}-\x_{0}\right)\right]^{2} \\
&-\frac{\eta_{3}}{8\eta_{1}\left(n \eta_{3}+\eta_{1}\right)} \sum_{l=1}^{m} \lambda_{l}\left\|\x_{l}-\x_{0}\right\|_2^2\left\|\x_{i}-\x_{0}\right\|_2^2 \\
&-\frac{\eta_{4}}{4 n \eta_{3}+4 \eta_{1}} c\left\|\x_{i}-\x_{0}\right\|_2^2+c+\left(\x_{i}-\x_{0}\right)^{\top} \g,\ i=1,\ldots,m.
\end{aligned}
\end{equation}

Then we can obtain the following system of equations since $\x_t$ is dropped and replaced by $\x_{\text{new}}$ at the $k$-th iteration, and $Q_k(\x_i)-Q_{k-1}(\x_i)=f(\x_i)-Q_{k-1}(\x_i)$, for any $\x_i$ in the current interpolation set.
\begin{equation}\label{system1-1}
{\overbrace{
\begin{pmatrix}
 \mathbf{A} & 
  \mathbf{J}
  &\mathbf{X}  \\
\mathbf{J}^{\top}   & \frac{n \eta_{4}^{2}}{2 n \eta_{3}+2 \eta_{1}}-2\eta_5 & \mathbf{0}_n^{\top}  \\
\mathbf{X}^{\top} & \mathbf{0}_n &  -2\eta_2\mathbf{I}
\end{pmatrix}}^{m+1+n}}
\left(\begin{array}{c}
\contour[2]{black}{$\lambda$} \\
c \\
\g
\end{array}\right)=
\left(\begin{array}{c}
0\\
\vdots \\
0\\
f(\x_\text{new})-Q_{k-1}{(\x_\text{new})}\\
0\\
\vdots \\
0
\end{array}\right),
\end{equation}
where \(\x_{\text{new}}\) denotes the new interpolation point, and $\contour[2]{black}{$\lambda$}=(\lambda_1,\ldots,\lambda_m)^{\top}$. Besides, $\mathbf{I}\in{\Re}^{n\times n}$ is the identity matrix, and \(\mathbf{A}\) has the elements
\begin{equation*}
\mathbf{A}_{i j}=\frac{1}{{8}\eta_1}\left[\left({\x}_{i}-{\x}_{0}\right)^{\top}\left({\x}_{j}-{\x}_{0}\right)\right]^{2}-\frac{\eta_{3}}{8\eta_{1}\left(n \eta_{3}+\eta_{1}\right)} \left\|\x_{i}-\x_{0}\right\|_2^2\left\|\x_{j}-\x_{0}\right\|_2^2,
\end{equation*}
where \(1 \leq i, j \leq m\). Besides, it holds that
\(
\mathbf{X}=\left(\x_1-\x_0, \x_2-\x_0,\ldots,\x_m-\x_0\right)^{\top},
\) 
and
\begin{equation*}
\mathbf{J}=
\left(
  1-\frac{\eta_{4}}{4 n \eta_{3}+4 \eta_{1}}\left\|\x_{1}-\x_{0}\right\|_2^2,
  \ldots, 
 1-\frac{\eta_{4}}{4 n \eta_{3}+4 \eta_{1}}\left\|\x_{m}-\x_{0}\right\|_2^2
 \right)^{\top}.
 \end{equation*}
We denote the matrix on the left hand side of (\ref{system1-1}) as KKT matrix \(\mathbf{W}\).

The quadratic model function \(Q(\x)\) is obtained from \(\contour[2]{black}{$\lambda$}, c,\g\) according to the expression of $Q(\x)$, (\ref{Q(x)G}) and (\ref{Tiswhat}). We can find that the least Frobenius norm updating model is a special case of the least \(H^2\) norm updating model in the following.

\begin{remark}
If \(C_{1}=C_2=0,C_{3}=1\), then $\eta_1=1, \eta_2=\eta_3=\eta_4=\eta_5=0$, and the KKT matrix is 
\begin{equation}\label{powellw}
\mathbf{W}=\left(
\begin{array}{cccc}
\bar{\mathbf{A}} & \mathbf{E} & \mathbf{X} \\
\mathbf{E}^{\top} & 0 & \mathbf{0}_n^{\top} \\
\mathbf{X}^{\top} & \mathbf{0}_{n} & \mathbf{0}_{nn} 
\end{array}\right),
\end{equation}
where $\mathbf{E}\in{\Re}^{m\times 1}$ is $(1,\ldots,1)^{\top}$, and \(
\bar{\mathbf{A}}_{i j}=\frac{1}{8}\left[\left({\x}_{i}-{\x}_{0}\right)^{\top}\left({\x}_{j}-{\x}_{0}\right)\right]^{2}, \ 1 \leq i, j \leq m.
\)

In this case, the Hessian matrix, corresponding to interpolation points $\x_1,\ldots,\x_m$, is 
\begin{equation}
\label{Old-KKT-Frob}
\mathbf{G}=\frac{1}{4}\sum^{m}_{i=1}\lambda_i (\x_i-\x_0)(\x_i-\x_0)^{\top}.
\end{equation}
The $(m+n+1)\times(m+n+1)$ matrix in (\ref{powellw}) is exactly the corresponding KKT matrix of least Frobenius norm updating  quadratic model \cite{NEWUOA}. Notice that the coefficient $\frac{1}{4}$ in (\ref{Old-KKT-Frob}) depends on the coefficient in Lagrange function (\ref{lagfunction}). 
\end{remark}

To reduce the computation complexity, we will apply the updating formula of the KKT inverse matrix in the following. 
Before discussing the inverse matrix of the KKT matrix, we firstly give a theorem here. The following theorem illustrates the condition in which the KKT matrix is invertible.

\begin{theorem}\label{invertible}
The $(m+n+1)\times (m+n+1)$ matrix \(\mathbf{W}\) is an invertible matrix if and only if the $(m+1) \times (m+1)$ matrix 
\begin{equation*}
\begin{pmatrix}
 \mathbf{A}+\frac{1}{2\eta_2}\mathbf{X}\mathbf{X}^{\top} & 
  \mathbf{J}\\
\mathbf{J}^{\top}   & \frac{n \eta_{4}^{2}}{2 n \eta_{3}+2 \eta_{1}}-2\eta_5 
\end{pmatrix}
\end{equation*}
is invertible. 
\end{theorem}

\begin{proof}
We know that
\begin{align*}
\begin{pmatrix}
 \mathbf{A} & 
  \mathbf{J}
  &\mathbf{X}  \\
\mathbf{J}^{\top}   & \frac{n \eta_{4}^{2}}{2 n \eta_{3}+2 \eta_{1}}-2\eta_5 & \mathbf{0}_n^{\top}  \\
\mathbf{X}^{\top} & \mathbf{0}_n &  -2\eta_2\mathbf{I}
\end{pmatrix}\rightarrow
\begin{pmatrix}
 \mathbf{A}+\frac{1}{2\eta_2}\mathbf{X}\mathbf{X}^{\top} & 
  \mathbf{J}
  &\mathbf{X}  \\
\mathbf{J}^{\top}   & \frac{n \eta_{4}^{2}}{2 n \eta_{3}+2 \eta_{1}}-2\eta_5 & \mathbf{0}_n^{\top}  \\
\mathbf{0}_{mn} & \mathbf{0}_n &  -2\eta_2\mathbf{I}
\end{pmatrix},
\end{align*}
where the arrow denotes elementary transformations, and $\mathbf{0}_{mn}\in{\Re}^{m\times n}$ is the zero matrix. Then the conclusion holds.
\end{proof}

Theorem \ref{invertible} gives the sufficient and necessary condition in which the KKT matrix is invertible. We call the inverse matrix of the KKT matrix as KKT inverse matrix for simplicity. 

Obtaining the parameters of the quadratic model function, i.e., \(\contour[2]{black}{$\lambda$},c,\g\), by solving KKT equations (\ref{system1-1}) at each iteration is not efficient, of which the computational complexity is $\mathcal{O}((m+n)^3)$. We apply the updating formula of the KKT inverse matrix with low computational complexity. 
Similarly to the discussion in Powell's work \cite{NEWUOA,powell04onupdating}, a natural question is what will happen on the KKT matrix with the iteration increasing. When the interpolation set is updated, we will find that only the \(t\)-th column and \(t\)-th row of \(\mathbf{W}\) change, since only $\x_t$ is dropped and replaced by $\x_\text{new}$ in the algorithm. Therefore, we apply the updating formula of the KKT inverse matrix given by Powell \cite{powell04onupdating}.

We define the vector $\contour[2]{black}{$\omega$}\in{\Re}^{m+n+1}$ with components $\omega_{i}$ that are separately equal to
\begin{equation}
\label{omega}
\left\{
\begin{aligned}
&\frac{1}{8\eta_1}\left[\left({\x}_{i}-{\x}_{0}\right)^{\top}\left({\x_{\text{new}}}-{\x}_{0}\right)\right]^{2}-\frac{\eta_{3}}{8\eta_{1}\left(n \eta_{3}+\eta_{1}\right)} \left\|\x_{i}-\x_{0}\right\|_2^2\left\|\x_{\text{new}}-\x_{0}\right\|_2^2, \\
&\quad \quad \quad \quad \quad \quad \quad \quad \quad \quad \quad \quad \quad \quad \quad \quad \quad \quad \quad \quad \quad \quad \quad \quad \quad \quad \text{ if \(1\le i\le m\)},\\
&1-\frac{\eta_{4}}{4 n \eta_{3}+4 \eta_{1}}\left\|\x_{\text{new}}-\x_{0}\right\|_2^2, \text{ if \(i= m+1\)},\\
&\left(\x_{\text{new}}-\x_{0}\right)_{i-m-1}, \text{ if \(m+2\le i\le m+n+1\)},
\end{aligned}
\right.
\end{equation}
and if an invertible KKT matrix \(\mathbf{W}\)'s \(t\)-th column and \(t\)-th row are replaced by the vector \(\contour[2]{black}{$\omega$}\) and  \(\contour[2]{black}{$\omega$}^{\top}\) separately, and the new matrix is denoted as \(\mathbf{W}_{\text{new}}\), and \(\mathbf{H}_{\text{new}}:=\mathbf{W}_{\text{new}}^{-1}\), \(\mathbf{H}:=\mathbf{W}^{-1}\), then we can obtain the following updating formula applied in our algorithm with least $H^2$ norm updating quadratic model.
\begin{equation}
\label{updateformula-t}
\begin{aligned}
\mathbf{H}_{\text{new}}=&\mathbf{H}+\sigma^{-1}\left\{\alpha\left({\e}_{t}-\mathbf{H} {\contour[2]{black}{$\omega$}}\right)\left({\e}_{t}-\mathbf{H} {\contour[2]{black}{$\omega$}}\right)^{\top}-\beta \mathbf{H} {\e}_{t} {\e}_{t}^{\top} \mathbf{H}\right.\\
&\left.+\tau\left[\mathbf{H} {\e}_{t}\left({\e}_{t}-\mathbf{H} {\contour[2]{black}{$\omega$}}\right)^{\top}+\left({\e}_{t}-\mathbf{H} {\contour[2]{black}{$\omega$}}\right) {\e}_{t}^{\top} \mathbf{H}\right]\right\},
\end{aligned}
\end{equation}
where
\begin{equation}\label{coeff}
\left\{
\begin{aligned}
&{\alpha}=\e_{t}^{\top} \mathbf{H} \e_{t},\ {\beta}=\frac{1}{8\eta_1}\left\|{\x_{\text{new}}}-{\x}_{0}\right\|_2^{4}-\contour[2]{black}{$\omega$}^{\top}\mathbf{H}\contour[2]{black}{$\omega$},\\
&{\tau}=\e_{t}^{\top} \mathbf{H}\contour[2]{black}{$\omega$},\ {\sigma}={\alpha} {\beta}+\tau^2,
\end{aligned}
\right.
\end{equation}
and \(\e_{t}\) is \((0,\ldots,0,1,0,\ldots,0)^{\top}\), of which only the $t$-th element is 1.

We obtain the new KKT inverse matrix  \(\mathbf{H}_\text{new}\) according to the updating formula (\ref{updateformula-t}), and then obtain \((\contour[2]{black}{$\lambda$},c,\g)^{\top}\) by
\begin{equation*}
\left(
\contour[2]{black}{$\lambda$},c,\g
\right)^{\top}=\mathbf{H}_\text{new}
\left(
0,\ldots,0,f(\x_\text{new})-Q_{k-1}{(\x_\text{new})},0,\ldots,0
\right)^{\top}.
\end{equation*}
The corresponding updating of $\mathbf{H}$ can be calculated in $\mathcal{O}((m+n)^2)$ operations.

After giving the updating formula, we consider more details to improve the robust property of the updating formula by choosing a new iteration point when using the algorithms based on least $H^2$ norm updating model function with KKT inverse matrices updated by (\ref{updateformula-t}).

The updating formula (\ref{updateformula-t}) of the inverse matrix of the KKT matrix has a denominator 
\(
{\sigma}={\alpha} {\beta}+\tau^2
\),  
and the expressions of \(\alpha, \beta\) and \(\tau\) are in (\ref{coeff}). In order to avoid the absolute value of the denominator \(\sigma\) from being too small, which will make the numerical updating instable, we use the model improvement step as Step 4 in Algorithm \ref{algo-TR}, i.e., reaching the iteration point $\x_{\text{opt}}+\dd$, where $\dd \in {\Re}^{n}$ is obtained by solving the problem
\begin{equation}\label{max-1}
\max_{\dd}\ \vert{\alpha} {\beta}+\tau^2\vert,\ 
\text { subject to }\ \|\dd\|_2 \leq \Delta_k.
\end{equation}
The objective function in (\ref{max-1}) is a function of \(\dd\) because \(\x_{\text{new}}=\x_{\text{opt}}+\dd\), where \(\x_{\text{opt}}\) is the minimum point at the current step. Subproblem (\ref{max-1}) is a quartic problem of \(\dd\). The model improvement step is chosen as the solution of a quartic problem on the trust region, and the truncated conjugated gradient method is applied to solve such subproblem, when the updating formula (\ref{updateformula-t}) is possibly not robust caused by an ill-poised interpolation set, in the implementation of the algorithm based on least $H^2$ norm updating quadratic model. 

In the current implementation, the algorithm will not accept the model, and then it will apply model improvement step, if $\rho_k<\hat{\eta}_1$ and the farthest point $\x_{\text{far}}$ from $\x_{\text{opt}}$ satisfies that $\|\x_{\text{far}}-\x_{\text{opt}}\|_2>2\Delta_k$, which is a similar criterion with that of Powell's solver NEWUOA. Besides, it changes the base point $\x_0$ to be the current $\x_{\text{opt}}$, which is the center of the next trust region, when $\|\x_{k}-\x_{0}\|_2>10\Delta_k$. Powell discusses more details \cite{NEWUOA}. The updating formula (\ref{updateformula-t}) indeed reduces the whole computation complexity during the iteration. More details of the choice of $\x_0$ can be seen in the work of Zhang \cite{Zhang2014}. In addition, for the geometry property and poisedness of the interpolation set, more details are in the work of Conn, Scheinberg and Vicente \cite{conn2008geometry}.

\section{Numerical Results\label{section7}}

To illustrate the advantages of our least \(H^2\) norm quadratic model, we present numerical results related to solving unconstrained  derivative-free optimization problems as (\ref{dfo-uncon}). Numerical experiments contain three parts, and they are separately the observation and comparison of interpolation error and stable updating, a simple simulation, and the performance profile based on solving  test problems. We implement a derivative-free trust-region algorithm according to the framework shown as Algorithm \ref{algo-TR} in Python for numerical test. The least $H^2$ norm quadratic model used in the test implementation is obtained by updating formula (\ref{system1-1}), and we update the KKT inverse matrix using formula (\ref{updateformula-t}). The model improvement step in the algorithm is obtained by solving subproblem (\ref{max-1}). To achieve a direct and fair comparison for different model functions in Example \ref{example-1}, Example \ref{example-2} and the comparison with least Frobenius norm updating quadratic model by performance profile, corresponding formulas will be substituted by that of least Frobenius norm updating model, and the framework is the same one.  The weight coefficients \(C_1, C_2,C_3\) are set equally as $\frac{1}{3}$ in the numerical experiments. In our numerical implementation, the radius $r$ is chosen as $\max \left\{10 \Delta_k, \max_{\x \in \mathcal{X}_k}\|\x-{\x}_{\text{opt}}\|_{2}\right\}$ at the $k$-th step, which is as same as the setting in the work of Zhang \cite{Zhang2014}\footnote{There are other ways to define $r$. Our definition is simple and enough for numerical experiment.}. One can obtain test codes for constructing least $H^2$ norm updating quadratic model from the online repository\footnote{https://github.com/PengchengXieLSEC/least\_H2\_norm\_updating\_model}. Numerical results confirm the advantage of our choice to obtain the quadratic model function by minimizing the $H^2$ norm instead of minimizing the Frobenius norm.

\subsection{Interpolation error and stable updating}

We make a numerical observation about the interpolation error and the stability when updating the quadratic model based on minimizing the \(H^{2}\) norm of the difference between two models.

In order to illustrate the advantages of applying \(H^{2}\) norm to obtain our model function, we use the following example to numerically capture the differences between the interpolation with least \(H^{2}\) norm and the interpolation with least Frobenius norm of the Hessian matrix. 

\begin{example}\label{example-1}
The updating (\ref{H2example}) can be transformed to obtaining $D_k$ by solving 
\begin{equation}\label{H2example-2}
\begin{aligned}
\mathop{\min}\limits_{D\in\mathcal{Q}}&\ \Vert D\Vert^2_{H^{2}{(\mathcal{B}^{r}_{2}(\x_{0}))}}\\
\text{subject to}&\ D(\x_{\text{new}})=f(\x_{\text{new}})-Q_{k-1}(\x_{\text{new}}),\ \x_{\text{new}}\in \mathcal{X}_k,\\
&\ D(\x_i)=0,\ \x_i\in \mathcal{X}_k\backslash\{\x_{\text{new}}\},
\end{aligned}
\end{equation}
where $D_k=Q_k-Q_{k-1}$. In this simple 2-dimensional example, the function $f-Q_{k-1}$ in problem (\ref{H2example-2}) is assumed to satisfy 
\begin{equation}
\label{f-example-1-during-iteration}
f(\x)-Q_{k-1}(\x)=
\left\{
\begin{aligned}
&1,\ \text{if}\ \x= \x_{\text{new}},\\
&0,\ \text{otherwise}, 
\end{aligned}
\right.
\end{equation}
during the iteration. 
\begin{remark}\color{black}
This example is exactly related to a Lagrange basis function. Notice that the old point $\x_t$ is dropped and replaced by $\x_{\text{new}}$ at the corresponding $k$-th iteration. The function $Q_k=Q_{k-1}+D_k$ is exactly the $k$-th model function, and the initial model is $Q_0(\x)=0$. Before going into the iteration, $f((0,0)^{\top})=1$, and $f(\x)=0,\forall\ \x\ne (0,0)^{\top}$, and then $f$ is defined by (\ref{f-example-1-during-iteration}) in the following steps. 
\end{remark} 
We use 3 interpolation points at each step, and the initial interpolation points of this simple example are \(\x_1=(0,0)^{\top}\), \(\x_2=(1,0)^{\top}\), \(\x_3=(0,1)^{\top}\), which follows a frequently-used choice.  The total number of the iteration is 3, and the trust region here is set as $\mathcal{B}_2^1(\x_{\text{opt}})$, where $\x_{\text{opt}}$ is the minimum point in the current iteration. We want to focus on the initial interpolation behavior comparison between least Frobenius norm updating model and least $H^2$ norm updating model without model improvement. It can be regarded as a sufficiently fair and intuitive comparison to compare the interpolation errors at the first 3 iterations with the fixed trust-region radius, which also supports the following simple choice of the region containing grid points being calculated interpolation errors on. The number of interpolation points is selected as 3 to ensure that the compared least Frobenius norm updating model is fully linear.

Fig. \ref{sub-zong-2} shows numerical results. In each sub-figures of  Fig. \ref{sub-zong-2}, we plot two lines to separately denote Powell's model and our model, which are separately updated based on least Frobenius norm updating and least \(H^{2}\) norm updating. Fig. \ref{sub-a} and Fig. \ref{sub-b} present the maximum interpolation error and the mean of interpolation error at all iteration points versus iterations. Fig. \ref{sub-c} and Fig. \ref{sub-d} present the maximum interpolation error and the mean of interpolation error at grid points versus iterations. The interpolation error at iteration points and the interpolation error at grid points are defined as \(\text{Err}_{\text{itr}}(\x)=\vert f(\x)-Q_k(\x)\vert\), and \(\text{Err}_{\text{grid}}(\y_{pq})=\vert f(\y_{pq})-Q_k(\y_{pq})\vert\), where \(\x\) is a historical iteration point, and \(\y_{pq}=(\frac{p}{100},\frac{q}{100})^{\top}\), \(p,q\in [-100,100]\cap\mathbb{Z}\).

The difference of the  interpolation error in Fig. \ref{sub-zong-2} shows  the advantage of our least \(H^{2}\) norm quadratic model. Notice that during the iteration, our model  has less interpolation error at the old dropped interpolation points and the grid points in the region $[-1,1]^2$ than least Frobenius norm updating model. In the other word, we can observe that least $H^2$ norm updating is stable, which is a consequent result of the property of $H^2$ norm itself and the projection relationship (\ref{projH}).

\begin{figure}[htbp]
\centering
\subfigure[\label{sub-a}]{
\centering
\includegraphics[height=1.2in]{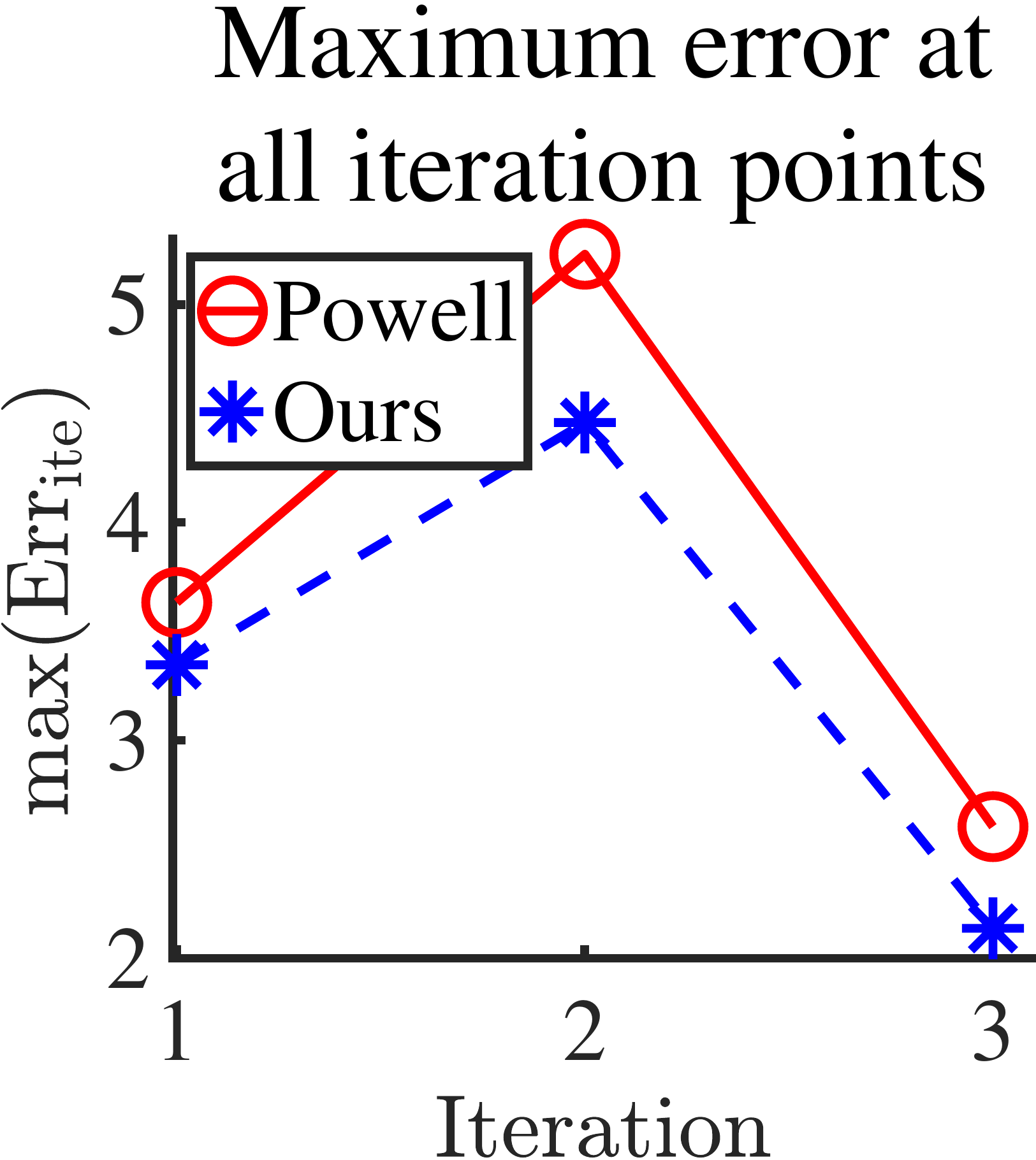}
}
\subfigure[\label{sub-b}]{
\centering
\includegraphics[height=1.2in]{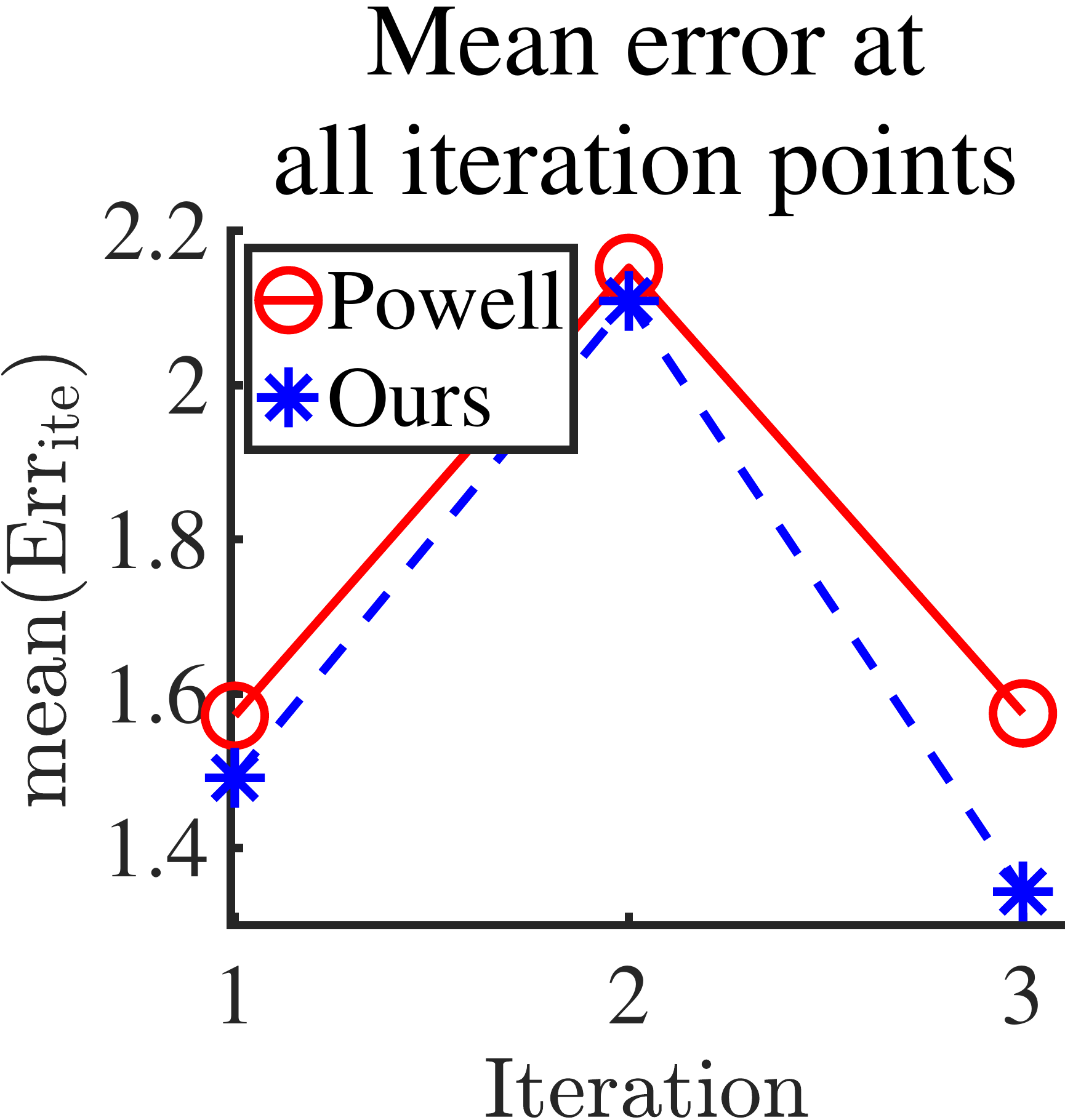}
}
\subfigure[\label{sub-c}]{
\centering
\includegraphics[height=1.2in]{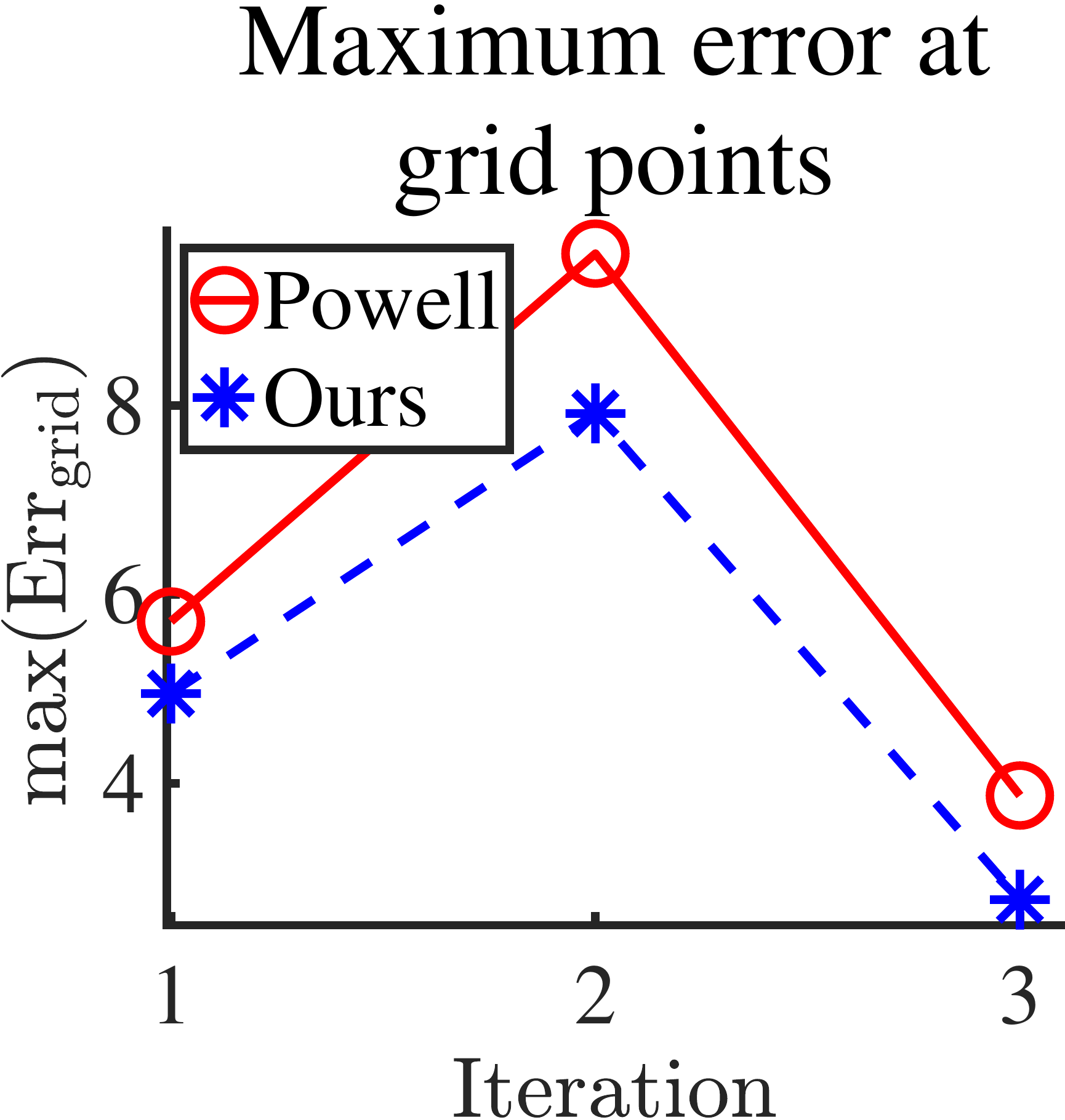}
}
\subfigure[\label{sub-d}]{
\centering
\includegraphics[height=1.2in]{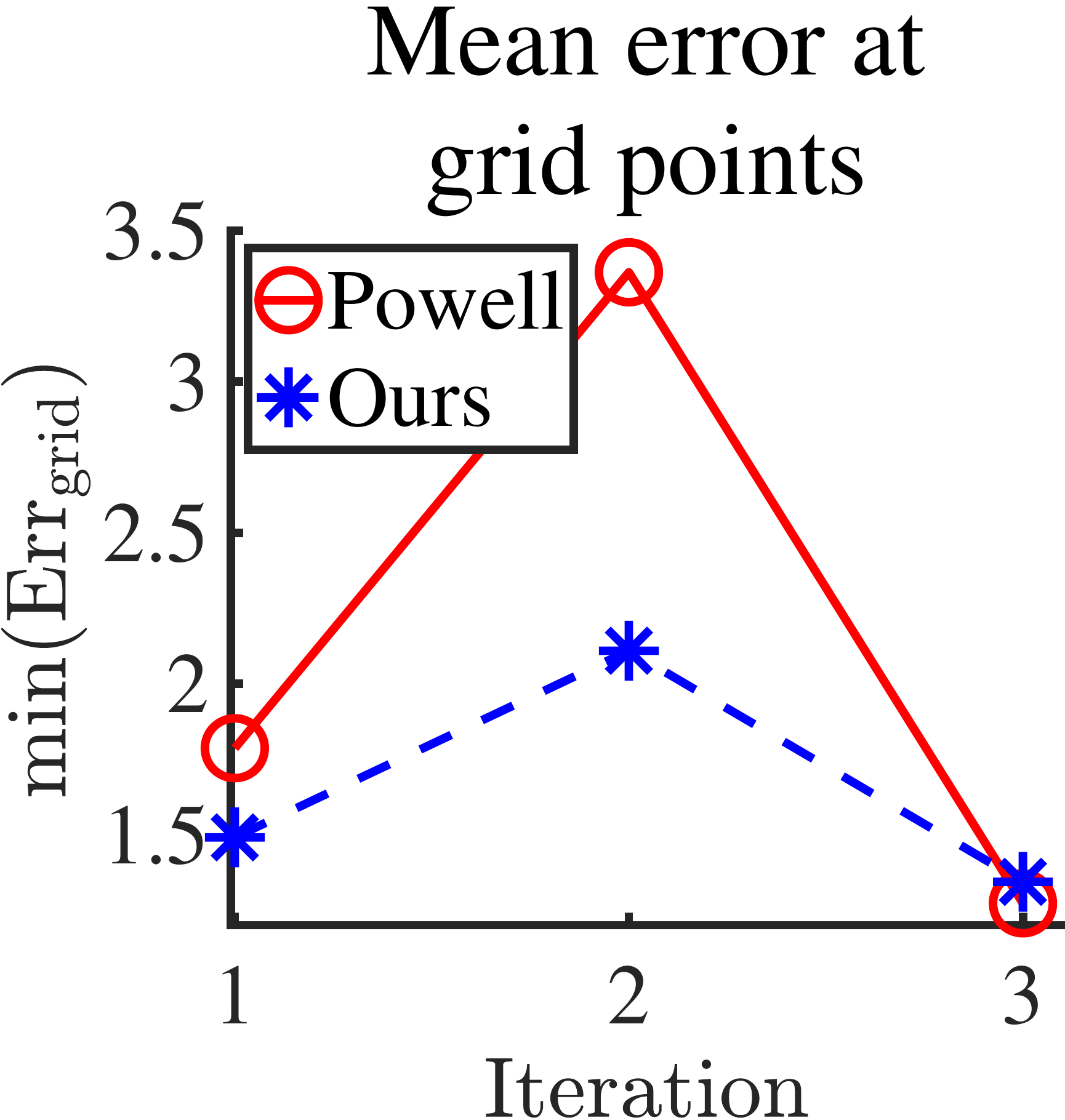}
}
\centering

\caption{Interpolation error comparison of different interpolation quadratic models.\label{sub-zong-2}}
\end{figure}

It can be noted that the objective function we desgin here scales the function value of the interpolation constraint at $\x_{\text{new}}$ to be 1, and the function $f$ itself is discontinuous and highly-nonlinear. Therefore, it is not necessary that the interpolation error is always in a small-scale or vanishing, since the model function is continuous, and the setting above can help us take a clear observation under fair and simple conditions.
\end{example}

\subsection{A simple simulation}

The following example simulates advantages of least $H^2$ norm updating model when iteratively solving a simple and classical test problem.
\begin{example}\label{example-2}
The objective function of this simple simulation is 2-dimensional Rosenbrock function 
\(
f(\x)=f(x_{(1)},x_{(2)})=(1-x_{(1)})^2+100(x_{(2)}-x_{(1)}^2)^2,
\) 
where $x_{(1)}$ and $x_{(2)}$ denote the 1st and 2nd components of the variable $\x$. The initial  interpolation points used in this simulation are uniformly distributed on the unit circle, i.e., \(\x_1=(0, 0)^{\top}\), \(\x_2=(\frac{\sqrt{3}}{2},\frac{1}{2})^{\top}\),  \(\x_3=(-\frac{\sqrt{3}}{2},\frac{1}{2})^{\top}\), \(\x_4=(0,-1)^{\top}\), which are simple and general settings when there are 4 interpolation points in the space ${\Re}^2$, and such interpolation set is poised and related to a minimum positive basis with uniform angles. Being different from the comparison of interpolation error in Example \ref{example-1}, we aim to compare the two models from the perspective of minimizing a classical function. The number of  interpolation points at a step is chosen according to Powell's suggestion \cite{NEWUOA} that \(n+2\) interpolation points can provide at least a constraint on the Hessian matrix of the least Frobenius norm updating model when minimizing an $n$-dimensional objective function. 

{\color{black}Firstly, we calculate one step in the trust region $\mathcal{B}_2^1(\x_1)$, and we can see that the minimum point of our model function is closer to the true one than least Frobenius norm with the same given interpolation points.} Comparison results of the function values at the minimum points of the least Frobenius norm updating quadratic model $Q_1$ and the least $H^2$ norm updating model $\bar{Q}_1$, we say $f(\x_{\text{opt}})$ and $f(\bar{\x}_{\text{opt}})$, are in Table \ref{para-error}, and we can see that \(f(\bar{\x}_{\text{opt}})<f({\x}_{\text{opt}})\). In Table \ref{para-error}, $\g_1$, $\mathbf{G}_1$ and $\bar{\g}_1$, $\bar{\mathbf{G}}_1$ separately denote the gradient and Hessian matrix of the model function \(Q_1\) and \(\bar{Q}_1\) after the first iteration.

\begin{table}[htbp]
  \caption{Parameters and some results in Example \ref{example-2}\label{para-error}}
  \setlength\tabcolsep{3.3pt}
  \centering
  \begin{tabular}{ccc}
    \toprule
    \(\x_{\text{opt}}=   (0.0263,0.8158)^{\top}\) & \(f(\x_{\text{opt}})=   67.3882\)      & \(\g_1=(-2, -62)^{\top}\)        \\
    \(\bar{\x}_{\text{opt}} = (0.0321,0.6365)^{\top}\) & \(f(\bar{\x}_{\text{opt}})=   41.3190\) &      \(\bar{\g}_1=(-1.8065, -56)^{\top}\)  \\
     \(\mathbf{G}_1= \begin{pmatrix} 76 & 0 \\ 0 & 76 \end{pmatrix}\)  
     & 
     \(\bar{\mathbf{G}}_1 = \begin{pmatrix} 64 & -0.3871 \\ -0.3871 & 88 \end{pmatrix}\) & - \\
    \bottomrule
  \end{tabular}
\end{table}

Secondly, we minimize the 2-dimensional Rosenbrock function iteratively using the derivative-free trust-region methods based on least Frobenius norm updating model and least $H^2$ norm updating model separately. {\color{black}The} initial interpolation points are given above and the initial trust-region radius is still 1. Besides, the tolerences of trust-region radius and the gradient norm of the model are separately set as $10^{-8}$, and $\mu=0.1$. Parameters for updating trust-region radius are $\gamma=2, \hat{\eta}_1=\frac{1}{4}, \hat{\eta}_2=\frac{3}{4}$. Fig. \ref{result-Rosenbrock} shows the iteration results, and other details are in Table \ref{table-result-Rosenbrock}. NF denotes the number of function evaluations.  The algorithm based on our model has a faster numerical convergence than the one using least Frobenius norm updating model, especially when the function value is less than $10^{-3}$, and this depends on the higher approximate accuracy of our model. This simulation shows advantages of least $H^2$ norm updating model.

\begin{table}[htbp]
  \caption{Numerical results of minimizing Rosenbrock function\label{table-result-Rosenbrock}}
  \centering
  \setlength\tabcolsep{2.8pt}
  \begin{tabular}{cccccc}
  \toprule
model& NF & final $f$ value & model gradient norm & best point \\
\midrule
Powell& 67   & \(3.8672\times10^{-9}\) & 0.0015 & \((1.00005607, 1.00011483)^{\top}\)\\
Ours&   55   &  \(8.0639\times10^{-12}\)   & \(7.8587\times 10^{-6}\) & \((1.00000271, 1.00000535)^{\top}\)\\
 \bottomrule
  \end{tabular}
\end{table}
\begin{figure}[htbp]
\centering
\includegraphics[scale=0.265]{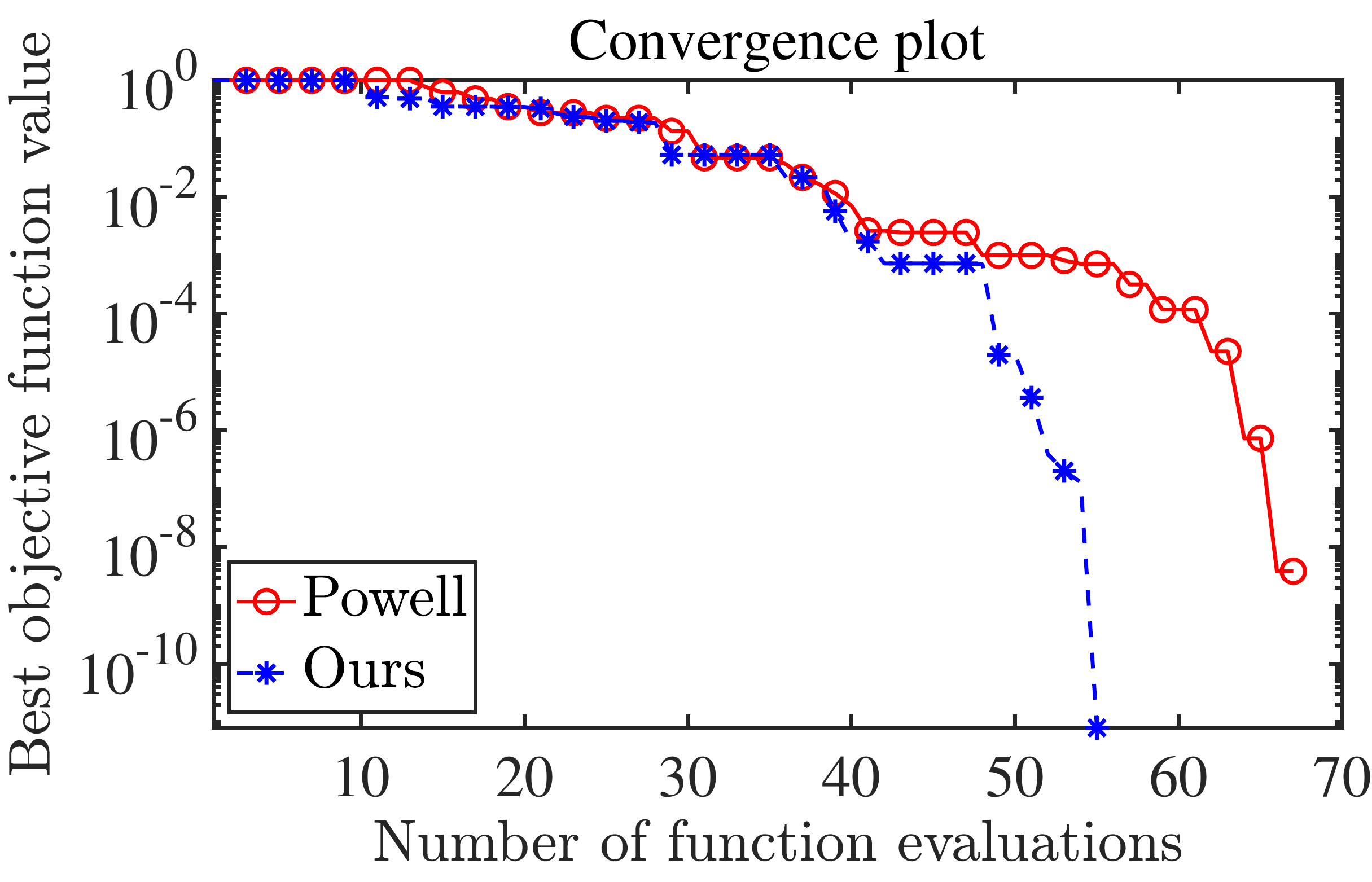}
\caption{Convergence plot of minimizing 2-dimensional Rosenbrock function based on Powell's least Frobenius norm updating model and our least $H^2$ norm updating model.\label{result-Rosenbrock}}
\end{figure}
In addition, we do numerical experiments with different number of interpolation points for each iteration of algorithm based on our least $H^2$ norm updating quadratic model, i.e., we separately set $m$ from $1$ to $\frac{1}{2}(n+1)(n+2)$.   Table \ref{new-func-diff} shows the number of function evaluations of applying algorithm based on least $H^2$ norm updating model functions, with different number of interpolation points, to minimize 2-dimensional Rosenbrock function. Other settings are as same as the settings above.

\begin{table}[htb]
  \caption{Minimizing Rosenbrock function with different number of interpolation points\label{new-func-diff}}
  \centering
  \begin{tabular}{ccccccc}
  \toprule
\multicolumn{6}{c}{Initial interpolation points}& NF \\
\midrule
\((0,0)^{\top}\)&-&-&-&-&-& 56 \\
\((0,0)^{\top}\) &\((1,0)^{\top}\)&-&-&-&-& 58 \\
\((0,0)^{\top}\)&\((1,0)^{\top}\)&\((0,1)^{\top}\)&-&-&-& 60 \\
\((0,0)^{\top}\)&\((\frac{\sqrt{3}}{2},\frac{1}{2})^{\top}\)&\((-\frac{\sqrt{3}}{2},\frac{1}{2})^{\top}\)&\((0,-1)^{\top}\)&-&-& 55 \\
\((0,0)^{\top}\)&\((1,0)^{\top}\)&\((0,1)^{\top}\)&\((-1,0)^{\top}\)&\((0,-1)^{\top}\)&-& 61 \\
\((0,0)^{\top}\)&\((1,0)^{\top}\)&\((0,1)^{\top}\)&\((-1,0)^{\top}\)&\((0,-1)^{\top}\)&\((\frac{\sqrt{2}}{2},-\frac{\sqrt{2}}{2})^{\top}\)& 63 \\
 \bottomrule
  \end{tabular}
\end{table}

Actually, basic experiments show that the best number of interpolation points at a step is not the same for problems with different structure. A primary consideration is that fewer interpolation points can lead to lower computational cost, and perhaps provide a more stable updating with the guarantee of the projection property given in Theorem \ref{projection}, since $\|Q_k-Q_{k-1}\|_{H^2(\Omega)}$ can be smaller when there are fewer interpolation constraints for $Q_k$ when solving problem (\ref{H2example}). Notice that the stability here is compared in the sense of different number of interpolation conditions, which is another claim beyond what Fig. \ref{sub-zong-2} shows. Moreover, an obvious fact is that the number of interpolation points can be selected according to the accuracy need of the interpolation approximation in the optimization process.

\end{example}

\subsection{Performance profile and data profile}

To further observe the numerical behavior of our algorithm based on least \(H^2\) norm updating quadratic model function, we try to solve the classical test problems and present the numerical results using the criterion called performance profile \cite{dolan2002benchmarking,audet2017derivative} and data profile \cite{BenchmarkingDFO,audet2017derivative}, which has been used for comparing derivative-free algorithms \cite{BenchmarkingDFO,audet2017derivative}. They are currently two of the most common way to compare several algorithms across a large test set of problems. We can capture information on efficiency (speed of convergence) and robustness (portion of problems solved) in a compact graphical format. The test problems with performance profile in Fig. \ref{fig-perf-four} and data profile in Fig. \ref{fig-data} are listed in Table \ref{table5}, of which the dimension is in the range of 10 to 100. They are all  from classical and common unconstrained optimization test functions collections \cite{powell04onupdating, Powell2003, Conn1994, Toint1978, Li2009, Luksan2010, Andrei2008, Li1988, CUTEr}.

\begin{table}[htb!]
  \centering   
  \setlength\tabcolsep{3.3pt}
    \caption{Test problems\label{table5}} 
      % \vspace{-0.4cm}
    \fontsize{8}{8}\selectfont
  \begin{tabular}{lllllll}  
    \toprule 
    ARGLINA& ARGLINB & ARGTRIG& BDQRTIC & BDALUE\\  BRYBND &CHAINWOO & CHEBYQAD & CHNROSNB & CHPOWELLS\\
    COSINECUBE& CURLY10 & CURLY20 & CURLY30 & DIXMAANE \\
     DIXMAANH & DIXMAANI &  DIXMAANK & DIXMAANO &DIXMAANP \\
     DQRTIC &ERRINROS &
    EXPSUM &EXTROSNB  & FLETCHCR \\
     FREUROTH  & GENROSE &
    INTEGREQ &MOREBV & NCB20 \\
     NONDQUAR &
    POWELLSG &   POWER &ROSENBROCK  &SBRYBND  \\
    SCOSINE &    SINQUAD &  SPARSINE &  SPHRPTS &SPMSRTLS \\
     SROSENBR & TOINTGSS & TQUARTIC   &WOODS &VARDIM \\ 
    \bottomrule   
  \end{tabular}
\end{table}

The performance profile and data profile depend on the numbers of iterations taken by all algorithms in an algorithm set \(\mathcal{A}\) to achieve a given accuracy when solving problems in a given problem set.

We define the value  
\[
f_{\mathrm{acc}}^{N}=({f(\x_{N})-f(\x_{\text{int}})})/({f(\x_{\text{best}})-f(\x_{\text{int}})}) \in [0,1],
\] 
and the tolerance \(\hat{\tau} \in [0,1]\), where \(\x_{N}\) denotes the best point found by the algorithm after \(N\) function evaluations, \(\x_{\text{int}}\) denotes the initial point, and \(\x_{\text{best}}\) denotes the best known solution {\color{black}obtained by all of the compared algorithms}. When \(f_{\mathrm{acc}}^{N} \ge 1-\hat{\tau}\),  we say that the solution reaches the accuracy \(\hat{\tau}\).

We denote \(N_{s,p}=\min\{n \in \mathbb{N},\ f_{\mathrm{acc}}^{n}\ge 1-\hat{\tau} \}\). $T_{s, p}=1$, if $f_{\mathrm{acc}}^{N} \geq 1-\hat{\tau}$ for some $N$; and $T_{s, p}=0$, if the solution of solver $s$ fails to reach the accuracy $\hat{\tau}$ on problem $p$ before the termination. 

Besides, we define 
\begin{equation*}
\begin{aligned}
&r_{s, p}=\left\{\begin{aligned}
&\frac{N_{s, p}}{\min \left\{N_{\tilde{s}, p}: \tilde{s} \in \mathcal{A}, T_{\tilde{s}, p}=1\right\}},\ \text {if} \ T_{s, p}=1, \\
&+\infty, \ \text {if} \ T_{s, p}=0,
\end{aligned}\right.
\end{aligned}
\end{equation*}
where \(s\) is the corresponding solver or algorithm. For the given tolerance \(\hat{\tau}\) and a certain problem \(p\) in the problem set \(\mathcal{P}\), the parameter \(r_{s, p}\) shows the ratio of the number of the function evaluations using the solver \(s\) divided by that using the fastest algorithm on the problem \(p\).

In the performance profile, \(\pi_{s}(\hat{\alpha})=\frac{1}{\vert\mathcal{P}\vert}\left\vert\left\{p \in \mathcal{P}: r_{s, p} \leq \hat{\alpha}\right\}\right\vert\), where \(\vert\cdot\vert\) denotes the cardinality. Generally, $\pi_{s}(\hat{\alpha})$ is the fraction of problems with a performance ratio $r_{p, s}$ bounded by $\hat{\alpha}$. Notice that a higher value of \(\pi_s(\hat{\alpha})\) represents solving more problems successfully. In particular, we know the following facts \cite{BenchmarkingDFO}.  
\begin{itemize}
\item[-] $\pi_{s}(1)$ is the fraction of problems for which solver $s \in \mathcal{A}$ performs the best (with the fewest number of evaluations, which is exactly the most important indicator for evaluating the convergance speed of derivative-free algorithms);

\item[-] $\pi_{s}(\hat{\alpha})$ is the fraction of problems solved by $s \in \mathcal{A}$, for a sufficiently large $\hat{\alpha}$. 
\end{itemize}

In addition, we use the data profile to provide raw information for a user with an expensive optimization problem, while the performance profiles focus on comparing different solvers. In particular, the data profile provides the number of function evaluations required to solve any of the problems, and they are useful to users with a specific computational budget who need to choose a solver that is likely to reach a given reduction in function value. In the data profile, 
$$
\delta_s(\hat{\beta})=\frac{1}{\vert \mathcal{P}\vert}\left\vert \left\{p \in \mathcal{P}: N_{s, p} \leq \hat{\beta}\left(n+1\right) T_{s, p}\right\}\right\vert,
$$
and a higher value of \(\delta_s(\hat{\alpha})\) represents solving more problems successfully.

\subsubsection{Efficiency comparison with least Frobenius norm updating quadratic model}

\label{Efficiency comparison with least Frobenius norm updating quadratic model}

Numerical results about the interpolation error and the simple simulation above show that the least $H^2$ norm updating qudratic model function is more accurate, and the numerical result in the following will show that such model accuracy can help the algorithms using our model converge faster than algorithms using least Frobenius norm updating quadratic model.

For each problem in this experiment, all algorithms start with an input point \(\x_{\text{int}}\), and the tolerance \(\hat{\tau}\) are separately set as \(10^{-1}, 10^{-3}, 10^{-5}\) and \(10^{-7}\). ``Powell'' denotes the algorithm with framework as Algorithm \ref{algo-TR} using Powell's least Frobenius norm updating quadratic model. ``Ours ($m=2n+1$)'' and ``Ours ($m=\lceil \frac{n}{2}\rceil+1$)'' both use least $H^2$ norm updating quadratic models, and share the same framework with ``Powell''. For the three algorithms in Fig. \ref{fig-perf-four}, the tolerances of the trust-region radius and the gradient norm  are set as $10^{-8}$ respectively. Their common initial trust-region radius is 1. {\color{black}The} parameters in Algorithm \ref{algo-TR} are $\gamma=2, \hat{\eta}_1=\frac{1}{4}, \hat{\eta}_2=\frac{3}{4}$, $\mu=0.1$. 
In order to achieve a fair comparison with other quadratic models, \(2n+1\) interpolation points are used at each iteration of methods ``Powell'' and ``Ours ($m=2n+1$)'', and they share the same initial interpolation points, $\x_{\text{int}}, \x_{\text{int}}\pm \e_{i}, i=1,\ldots,n$. Besides, \(\lceil \frac{n}{2}\rceil +1\) interpolation points are used at each iteration of method ``Ours ($m=\lceil \frac{n}{2}\rceil+1$)'', and the initial interpolation points of method ``Ours ($m=\lceil \frac{n}{2}\rceil+1$)'' are $\x_{\text{int}}, \x_{\text{int}}+ \e_{\lceil\frac{i}{2}\rceil}, i=1,\ldots,n$. 

In fact, as we discussed, choosing different $m$ for different problems has different performance, and it is worth being further studied in the future since least $H^2$ norm updating quadratic model already reduces the lower bound of the number of  interpolation conditions at each step. The performance of ``Ours ($m=\lceil \frac{n}{2}\rceil+1$)'' can show numerical advantages of using fewer interpolation points at each iteration.

\begin{figure}[htbp]
\centering
\includegraphics[scale=0.31]{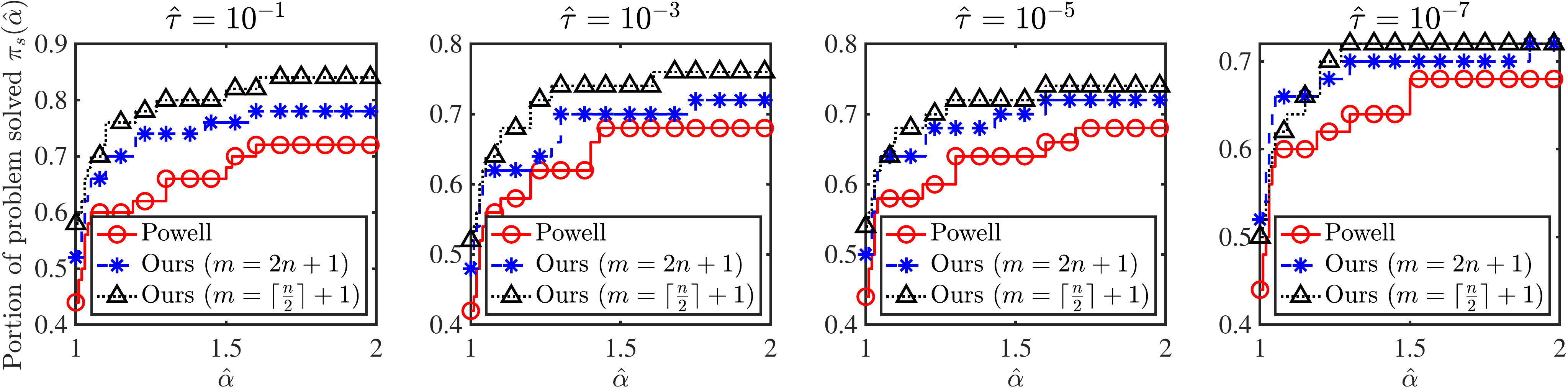}
  % \vspace{-0.4cm}
\caption{Performance profiles of solving test problems with derivaitve-free trust-region algorithms based on different quadratic models.\label{fig-perf-four}}
\end{figure}

In Fig. \ref{fig-perf-four}, values of $\pi_s(1)$ give the results that the number of problems that algorithm based on $H^2$ norm updating quadraic model with $m=\lceil \frac{n}{2}\rceil +1$ solved uses the fewest number of evaluations (including ties) in the cases $\hat{\tau}=10^{-1}, \hat{\tau}=10^{-3},\hat{\tau}=10^{-5}$, and the one with $m=2n+1$ uses the fewest number (including ties) in the case $\hat{\tau}=10^{-7}$. The two algorithms using least $H^2$ norm updating quadratic models are more efficient than the one using least Frobenius norm updating quadratic model.

\subsubsection{Numerical comparison with other derivative-free algorithms}

For a further numerical observation on the overall performance of our least $H^2$ norm updating quadratic model, we give the  comparison result with some matural and state-of-the-art algorithms using the data profile in the following. Compared derivative-free trust-region algorithms based on Powell's model (least Frobenius norm updating quadratic model) and least Frobenius norm model \cite{conn1997convergence} are separately the Python interface of NEWUOA \cite{NEWUOA} in PDFO\footnote{https://www.pdfo.net} and the Python implementation of the DFO algorithm\footnote{https://coral.ise.lehigh.edu/katyas/software}. Besides, Neader-Mead simplex algorithm and BFGS method with first-difference estimated derivatives are obtained from the scipy.optimize library\footnote{https://docs.scipy.org/doc/scipy}.  

For each problem in this experiment, all algorithms start with an input point \(\x_{\text{int}}\), and the tolerance \(\hat{\tau}\) is set as \(10^{-2}\). We keep the same settings of ``Ours ($m=2n+1$)'' and ``Ours ($m=\lceil \frac{n}{2}\rceil +1$)'' as that in Section \ref{Efficiency comparison with least Frobenius norm updating quadratic model}. For the trust-region algorithms ``NEWUOA'', and ``DFO-py'' in Fig. \ref{fig-data}, the tolerances of the trust-region radius and the gradient norm  are set as $10^{-8}$ respectively. Their common initial trust-region radius is 1. Besides, \(2n+1\) interpolation points are used at each iteration of  trust-region methods ``NEWUOA'' and ``DFO-py'', in our numerical experiments, and they share the same initial interpolation points with ``Ours ($m=2n+1$)''. 
For ``Nelder-Mead-py'', the initial simplex is an $n+1$ dimensional simplex with vertices $\x_{\text{int}}, \x_{\text{int}}+\e_{i}, i=1,\ldots,n$, and the bound of the absolute error of the function value between iterations is set as $10^{-8}$. 
For ``BFGS'', the relative step size to use for numerical approximation of the gradient is selected automatically by setting ``None'', and the gradient norm must be less than $10^{-8}$ before successful termination.

\begin{figure}[htbp]
\centering
\includegraphics[scale=0.28]{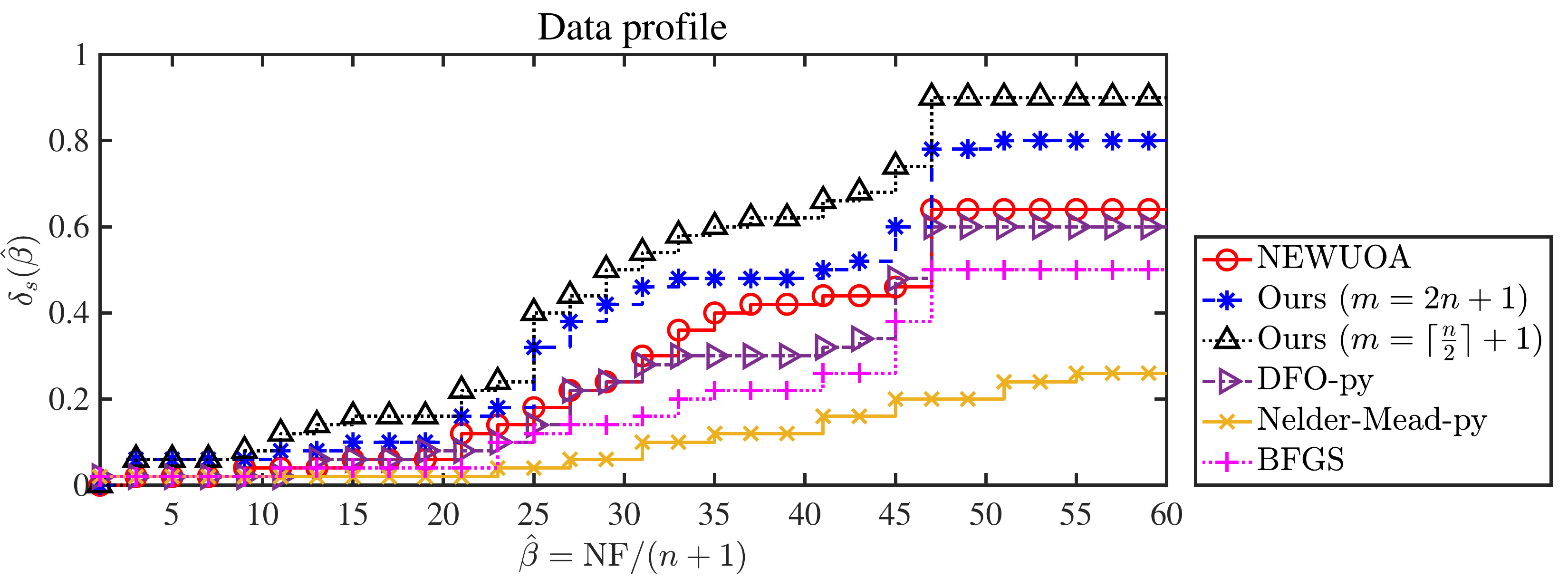}
\caption{Data profile of solving test problems with different algorithms.\label{fig-data}}
\end{figure}

We can observe from Fig. \ref{fig-data} that the trust-region algorithms based on least \(H^2\) norm updating quadratic model functions have better performance on efficiency and robustness than others, and they both can solve more than 75\% problems when \(\hat{\beta}\) is approximately larger than 46. The method using the model, formed by minimizing $H^2$ norm satisfying $m=\lceil\frac{n}{2}\rceil+1$ interpolation condition equations at each iteration, is more robust and efficient than the others, which can be seen from the values of $\delta_s(\hat{\beta})$. The results above show the efficiency and robustness advantages of the algorithm based on our model.

\section{Conclusions and Future Work\label{section8}}
\begin{sloppypar}
Under-determined interpolation of quadratic model functions can reduce interpolation points and the number of function evaluations for derivative-free trust-region optimization methods. In order to obtain a unique quadratic model function, it is a conventional method to determine the coefficients of the quadratic function in each iteration by solving an optimization problem constrained by interpolation conditions. This paper tries to obtain quadratic model function by minimizing the \(H^2\) norm of the difference between the new quadratic model function and the old one in the iteration, which reduces the lower bound of the number of interpolation points or equations. Projection property and the error bound are given. A corresponding updating formula to calculate the coefficients of model function is also derived based on the KKT condition of a convex optimization problem. Updating formulas above provide more choices for under-determined least norm updating quadratic models. Numerical results indicate the better performance of our algorithm from different perspectives. 
\end{sloppypar}

As a future work, more convergence properties of derivative-free trust-region algorithms based on least \(H^2\) norm updating models can be developed. One can also design adaptive weight coefficients for problems with different structure, and obtain the $k$-th quadratic model function by minimizing
\begin{equation*}
C^{(k)}_{1}\Vert Q-Q_{k-1}\Vert^2_{H^{0}(\Omega)}+C^{(k)}_{2}\vert Q-Q_{k-1}\vert^2_{H^{1}(\Omega)}+C^{(k)}_{3}\vert Q-Q_{k-1}\vert^2_{H^{2}(\Omega)}
\end{equation*}
in practice, where $k$ denotes the $k$-th  iteration step. Another possible future work is to look for better choices of the number of the interpolation points, since we have already reduced the lower bound of such number. Other types of interpolation models for derivative-free optimization can be studied on for the further research.

% \section*{Acknowledgments}
% We would like to acknowledge the assistance of volunteers in putting
% together this example manuscript and supplement.

\bibliographystyle{siamplain}
\bibliography{references}
\end{document}